\documentclass[12pt]{amsart}
\usepackage{amsbsy}
\usepackage{graphicx,epsfig,subfigure,psfrag}
\begin{document}

\newtheorem{theorem}{Theorem}
\newtheorem{proposition}{Proposition}
\newtheorem{lemma}{Lemma}
\newtheorem{corollary}{Corollary}
\newtheorem{definition}{Definition}
\newtheorem{remark}{Remark}
\newcommand{\tex}{\textstyle}
\numberwithin{equation}{section} \numberwithin{theorem}{section}
\numberwithin{proposition}{section} \numberwithin{lemma}{section}
\numberwithin{corollary}{section}
\numberwithin{definition}{section} \numberwithin{remark}{section}
\newcommand{\ren}{\mathbb{R}^N}
\newcommand{\re}{\mathbb{R}}
\newcommand{\n}{\nabla}
\newcommand{\p}{\partial}
\newcommand{\iy}{\infty}
\newcommand{\pa}{\partial}
\newcommand{\fp}{\noindent}
\newcommand{\ms}{\medskip\vskip-.1cm}
\newcommand{\mpb}{\medskip}
\newcommand{\AAA}{{\bf A}}
\newcommand{\BB}{{\bf B}}
\newcommand{\CC}{{\bf C}}
\newcommand{\DD}{{\bf D}}
\newcommand{\EE}{{\bf E}}
\newcommand{\FF}{{\bf F}}
\newcommand{\GG}{{\bf G}}
\newcommand{\oo}{{\mathbf \omega}}
\newcommand{\Am}{{\bf A}_{2m}}
\newcommand{\CCC}{{\mathbf  C}}
\newcommand{\II}{{\mathrm{Im}}\,}
\newcommand{\RR}{{\mathrm{Re}}\,}
\newcommand{\eee}{{\mathrm  e}}
\newcommand{\LL}{L^2_\rho(\ren)}
\newcommand{\LLL}{L^2_{\rho^*}(\ren)}
\renewcommand{\a}{\alpha}
\renewcommand{\b}{\beta}
\newcommand{\g}{\gamma}
\newcommand{\G}{\Gamma}
\renewcommand{\d}{\delta}
\newcommand{\D}{\Delta}
\newcommand{\e}{\varepsilon}
\newcommand{\var}{\varphi}
\newcommand{\lll}{\l}
\renewcommand{\l}{\lambda}
\renewcommand{\o}{\omega}
\renewcommand{\O}{\Omega}
\newcommand{\s}{\sigma}
\renewcommand{\t}{\tau}
\renewcommand{\th}{\theta}
\newcommand{\z}{\zeta}
\newcommand{\wx}{\widetilde x}
\newcommand{\wt}{\widetilde t}
\newcommand{\noi}{\noindent}
\newcommand{\uu}{{\bf u}}
\newcommand{\xx}{{\bf x}}
\newcommand{\yy}{{\bf y}}
\newcommand{\zz}{{\bf z}}
\newcommand{\aaa}{{\bf a}}
\newcommand{\cc}{{\bf c}}
\newcommand{\jj}{{\bf j}}
\newcommand{\ggg}{{\bf g}}
\newcommand{\UU}{{\bf U}}
\newcommand{\YY}{{\bf Y}}
\newcommand{\HH}{{\bf H}}
\newcommand{\GGG}{{\bf G}}
\newcommand{\VV}{{\bf V}}
\newcommand{\ww}{{\bf w}}
\newcommand{\vv}{{\bf v}}
\newcommand{\hh}{{\bf h}}
\newcommand{\di}{{\rm div}\,}
\newcommand{\ii}{{\rm i}\,}
\newcommand{\inA}{\quad \mbox{in} \quad \ren \times \re_+}
\newcommand{\inB}{\quad \mbox{in} \quad}
\newcommand{\inC}{\quad \mbox{in} \quad \re \times \re_+}
\newcommand{\inD}{\quad \mbox{in} \quad \re}
\newcommand{\forA}{\quad \mbox{for} \quad}
\newcommand{\whereA}{,\quad \mbox{where} \quad}
\newcommand{\asA}{\quad \mbox{as} \quad}
\newcommand{\andA}{\quad \mbox{and} \quad}
\newcommand{\withA}{,\quad \mbox{with} \quad}
\newcommand{\orA}{,\quad \mbox{or} \quad}
\newcommand{\atA}{\quad \mbox{at} \quad}
\newcommand{\onA}{\quad \mbox{on} \quad}
\newcommand{\ef}{\eqref}
\newcommand{\mc}{\mathcal}
\newcommand{\mf}{\mathfrak}

\newcommand{\ssk}{\smallskip}
\newcommand{\LongA}{\quad \Longrightarrow \quad}
\def\com#1{\fbox{\parbox{6in}{\texttt{#1}}}}
\def\N{{\mathbb N}}
\def\A{{\cal A}}
\newcommand{\de}{\,d}
\newcommand{\eps}{\varepsilon}
\newcommand{\be}{\begin{equation}}
\newcommand{\ee}{\end{equation}}
\newcommand{\spt}{{\mbox spt}}
\newcommand{\ind}{{\mbox ind}}
\newcommand{\supp}{{\mbox supp}}
\newcommand{\dip}{\displaystyle}
\newcommand{\prt}{\partial}
\renewcommand{\theequation}{\thesection.\arabic{equation}}
\renewcommand{\baselinestretch}{1.1}
\newcommand{\Dm}{(-\D)^m}

\title
{\bf Steady states, global existence and blow-up for fourth-order
semilinear parabolic equations of Cahn--Hilliard type}

\author{Pablo~\'Alvarez-Caudevilla and Victor~A.~Galaktionov}

\address{Universidad Carlos III de Madrid,
Av. Universidad 30, 28911-Legan\'es, Spain -- Work phone number: +34-916249099}
\email{pacaudev@math.uc3m.es}

\address{Department of Mathematical Sciences, University of Bath,
 Bath BA2 7AY, UK}
\email{vag@maths.bath.ac.uk}


\thanks{The first author is partially supported by the Ministry of Science and Innovation of
Spain under the grant MTM2009-08259.}

\date{\today}





\begin{abstract}

Fourth-order semilinear parabolic equations of the
Cahn--Hilliard-type
 \be
 \label{01}
 u_t + \D^2 u = \g u  \pm \D (|u|^{p-1}u) \inB \O \times
 \re_+,
  \ee
  are considered in a smooth bounded domain
   $\O \subset
  \ren$ with Navier-type boundary conditions
   on $\p \O$, or $\O = \ren$, where $p>1$ and $\g$ are given real parameters.
 The sign $``+"$ in the ``diffusion term" on the right-hand side  means the stable
 case, while $``-"$ reflects the unstable (blow-up) one,
 with the simplest, so called {\em limit}, canonical model for
 $\g=0$,
  \be
  \label{02}
  u_t + \D^2 u= \pm \D(|u|^{p-1}u) \inA.
   \ee

   The following three main problems
  are studied:

  (i) for the unstable model \ef{01}, with the $- \D (|u|^{p-1}u)$,
   existence and multiplicity of classic steady states in $\O \subset \ren$ and their global
  behaviour for large $\g>0$;

  (ii) for the stable model \ef{02},
  global existence of smooth solutions $u(x,t)$ in $\ren \times
  \re_+$ for bounded initial data $u_0(x)$ in the subcritical case
   $p \le p_{*}= 1 + \frac {4}{(N-2)_+}$; and

  (iii) for the unstable model \ef{02},  a relation between finite time blow-up and structure of
 regular and  {\em singular} steady states in the supercritical range.
In particular, three distinct families of Type I and II blow-up
patterns are
 introduced in the unstable case.

\end{abstract}

\maketitle




\section{Introduction and motivation for main problems: steady states, global existence,
 and blow-up}
 \label{S1}

\subsection{Models and preliminaries}

\noindent In this paper, we study some properties of solutions of
the following {\em fourth-order  parabolic equation of the
Cahn--Hilliard (C--H) type}:
\begin{equation}
\label{11}
    u_{t} +\Delta^2 u = \g u \pm\Delta(|u|^{p-1}u)\, \quad
    \hbox{in}\quad \O\times \re_+ \whereA p>1, \,\,\,\g \in \re,
\end{equation}
with homogeneous Navier-type boundary conditions and bounded
smooth initial data,
\begin{equation}
\label{12}
    u=\D u=0 \quad \hbox{on}\quad \p\O, \qquad u(x,0)=u_0(x) \quad \hbox{in}\quad \O.
\end{equation}
\noi We assume that $\O$ is a bounded domain of $\ren$, $N\geq 1$,
with smooth boundary $\p\O$ of class $\mathcal{C}^{2+\nu}$ for
some $\nu\in (0,1)$. Here, \eqref{11} is a {\em semilinear}
parabolic equation with the only nonlinearity entering as a
second-order diffusion-like operator. The sign $``+"$ in the
``diffusion term" on the right-hand side of \ef{11} corresponds to
the stable
 case, while $``-"$ reflects the unstable (blow-up) one.

\ssk

 Firstly, we obtain  existence and multiplicity results for the
steady-states of the unstable CH equation \eqref{11} based on a
combination of analytical methods. Namely, we use variational
methods, such as the fibering approach, and, based on potential
operators,  Lusternik--Schnirel'man category--genus theory, and
others, such as homotopy approaches or perturbation theory.  We
specifically obtain that, depending on the parameter $\g$, there
exists a different number of stationary solutions.


\ssk

 Secondly, using scaling blow-up methods, global existence and uniqueness
of global classical bounded solutions for  the stable
\emph{Cahn--Hilliard equation} \ef{11}, with the $+ \D
(|u|^{p-1}u)$, in $\ren \times \re_+$,
are shown to exist up to a critical exponent $p_*= 1 + \frac
{4}{N-2} \big(\equiv \frac {N+2}{N-2}=p_{\rm S}\big)$, $N \ge 3$,
or $p_*=+\iy$ for $N=1,\,2$.

\ssk

 Thirdly, in the last part of the paper, different types of blow-up
solutions  are analysed  for the unstable \emph{Cahn--Hilliard
equation} \ef{11}, with $- \D (|u|^{p-1}u)$, by
using the similarity profiles associated with  this unstable
equation. This methodology will provide us with a direct
connection with the previous analysis carried out for the
multiplicity of a variational problem. However, in this particular
case the problem is not variational so a homotopy/perturbation
analysis must be performed.

\ssk

Throughout this paper, we also state and leave several open
difficult mathematical problems for these nonlinear problems and
other similar ones.

\ssk

 There are a huge  amount of publications related to  equations
such as \eqref{11}. Among other models, the most popular and
detailed studied ones are
the Cahn--Hilliard and Sivashinsky-type equations.
 We refer to papers \cite{AP, UnCh, ANC1} and
 to surveys in \cite{GVSur02, EGW}, where necessary aspects
 of global existence and blow-up of solutions for \ef{11} are
 discussed in sufficient detail.

Particularly,
the Sivashinsky equation is analyzed in studying phase turbulence
in fluids, thermal instabilities of flame fronts, the directional
solidification in alloys or the interface instability during the
application of industrial beam cutting techniques.
As an example, let us mention an interesting result in the context
of directional solidification of a dilute binary alloy that appears in
Novick-Cohen--Grinfeld \cite{NC-G}, where the steady-states of the
\emph{Sivashinsky equation}
\begin{equation}
\label{app1}
    u_t = \D ( u^2 -u -\e^2 \D u)-\a u \withA \a>0,
\end{equation}
were analyzed, focusing, specifically, on the problem of
multiplicity of solutions, which is also one of  the main topics
in this paper.

Moreover, the classic \emph{Cahn--Hilliard equation}  describes
the dynamics of a pattern formation in phase transition in alloys,
glasses, and polymer solutions. This equation has been extensively
studied in the past years but many questions still remain
unanswered. See below discussions and details about applications
and characteristics of these Cahn--Hilliard equations--type.


\subsection{Main results of the paper and layout}

Sections \ref{S2} is devoted to preliminaries about Cahn--Hilliard
equations where discussions about applications and specific
analytical characteristics
 of these equations are carried out.

 In Section
\ref{S3}, we study smooth stationary solutions of \ef{11} via the
so called {\em fibering method}, obtaining existence  and
multiplicity results for such steady states of the problem
\eqref{11}. In other words, the solutions of the parameter
dependent semilinear elliptic equation
\begin{equation}
\label{15}
    -\Delta^2 u + \g u \pm \Delta (|u|^{p-1}u)=0\, \quad
    \hbox{in}\quad \O,
    \quad u= \D u=0 \quad \mbox{on} \quad \p \O,
\end{equation}
which are now regarded as steady-states of the evolution equation
\eqref{11}. In particular, we obtain that, depending on the value
of the parameter $\g$, the unstable Cahn-Hilliard equation
$\eqref{15}_{-}$ (with the minus sign in $\D$)  possesses one or several
solutions or no solutions at all. The results can be summarized
as follows:

\begin{itemize}
\item If the parameter $\g\leq  K \l_1$,  with $K>0$ a positive constant and $\l_1>0$ the first eigenvalue
of the bi-harmonic operator, i.e., $\D^2 \varphi_1=\l_1
\varphi_1$, then there exists at least one  solution for the
Cahn-Hilliard equation $\eqref{15}_{-}$; and;

\item When the parameter is greater than the  first eigenvalue of the bi-harmonic
equation $\l_1$, multiplied by the positive constant $K$, then
there will not be any solution at all, if one assumes only
positive solutions. However,
 for oscillatory solutions of changing sign the number of possible solutions increases with the value of the parameter $\g$. In fact, when the parameter
 $\g$ goes to infinity, one has an arbitrarily large number of distinct solutions.
\end{itemize}

In Section \ref{S5}, we return to the original parabolic problem
of the type \ef{11}, and now, we concentrate on the problem of
global existence and uniqueness of global classical bounded
solutions in the simplest canonical {\em limit stable
Cahn--Hilliard equation},
 \be
 \label{sCH}
  u_t= - \D^2 u + \D ( |u|^{p-1} u) \inB \ren \times \re_+,
   \quad u(x,0)=u_0(x) \inB \ren.
   \ee
Since our goal is to establish sufficient conditions of {\em
non-blow-up} of solutions at any point in the $\{x,t\}$-space, we
consider the Cauchy problem in $\ren \times \re_+$ with smooth
bounded initial data, i.e., we take $\O = \ren$. We prove, using a standard scaling method in Nonlinear PDE theory,  that
global unique solutions (non-blow-up solutions in finite time) of
\ef{sCH} exist in the subcritical (in fact, Sobolev) range
 \be
  \label{sub1}
   \tex{
    1<p \le p_*= 1 + \frac {4}{N-2} \big(\equiv \frac {N+2}{N-2}=p_{\rm S}\big), \quad N \ge 3
     \quad (p_*=+\iy\,\,\,\mbox{for}\,\,\,N=1,\,2),
     }
     \ee
showing the existence of uniform a priori bounds
 $$
 |u(x,t)| \le C \inB
\ren \times \re_+.
  $$
   In addition, by the same scaling technique,
we prove that the non-autonomous C--H equation with
$a(x)=|x|^\a>0$ for all $x \ne 0$, $\a>0$,
 \be
 \label{sCHa}
  u_t= - \D^2 u + \D (|x|^\a |u|^{p-1} u) \inB \ren \times \re_+,
   \quad u(x,0)=u_0(x) \inB \ren
   \ee
 {\em does not admit a localized blow-up at the origin $x=0$} in
a larger parameter range
 \be
  \label{sub1a}
   \tex{
    1<p \le p_*(\a)= 1 + \frac {2(\a+2)}{N-2}, \quad N \ge 3.
     }
     \ee
However, this does not prohibit a possible blow-up at some $x_0
\not = 0$, at which the range \ef{sub1} puts in charge again.
Moreover, in the supercritical range $p>p_*$, we observe that
those a priori bounds cannot be obtained through the techniques
used above for the parameter range \eqref{sub1}. Therefore, one
cannot avoid the possibility of existence of blow-up solutions in
this particular supercritical range.

\ssk

Concerning the related  {\em limit unstable C--H equation}
\begin{equation}
\label{ch6}
    u_{t} = - \D^2 u -\D(\left|u\right|^{p-1}u)\,,
\end{equation}
it is well known for a long period (see a blow-up survey
\cite{GVSur02}) that solutions can blow-up for any $p>1$.
Moreover, in general, there exists a countable family of various
self-similar blow-up solutions, which turned out to be positive at
the critical (Fujita) exponent \cite{EGW}
  \be
  \label{p01}
  \tex{
  p=p_0=1+ \frac 2N.
  }
  \ee
Such mass-conserving  solutions, which blow-up as $t \to T^-<\iy$,
can admit similarity extensions beyond, i.e., for $t>T$ (see
\cite{GalJMP} for the case \ef{p01}),
 though there remain some difficult open mathematical problems.
 This
can be compared to Leray's argument of 1934 for the extension
for $t>T$ of self-similar blow-up solutions (as $t \to T^-$) for
the Navier--Stokes equation in $\re^3$; see precise Leray's
statements, references, and related comments  in
\cite[\S~2.2]{GalNSE}\footnote{Since 1996, self-similar (Type-I)
blow-up for Navier--Stokes equations was ruled out (see main
references in \cite[\S~1.1]{GalNSE}), so an unknown and a more
complicated Type-II one seems to be necessary. Fortunately, for
the C--H equation \ef{ch6}, self-similar behaviour exists in both
limits $t \to T^\pm$.}.

Thus, in other words, for the {\em unstable} C--H models, the
critical exponent is $p_*=1$, so that the problem of existence and
nonexistence of  blow-up becomes irrelevant.

In Sections \ref{S7} and \ref{S.10}, as a unified issue concerning
studied above stationary and global solutions of \ef{ch6},
 we will try to connect possible blow-up of
solutions with some features of the structure of regular and {\em
singular} steady states, which are not bounded in $L^\iy$. The
first blow-up type solutions of \ef{ch6} under scrutiny in these
sections  are the ones obtained at the critical Sobolev exponent
defined as in \ef{sub1}.
    The main idea behind
this blow-up patterns is that the blow-up can occur via some kind
of ``slow" motion about its stationary solutions, to be explained
in detailed later on.

Furthermore, for the final type of blow-up patterns of the
unstable C--H equation, we shall adapt the techniques used to
obtain blow-up patterns
 for the nowadays classic {\em semilinear heat equation}
 $$
 u_t=-\D u +|u|^{p-1} u,
 $$
 in order to ascertain these final and new blow-up structures. To do so, we are required to
  construct special  spectral theory of linear rescaled
  operators involved. In particular,  such a spectral analysis
  requires {\em generalized Hermite polynomial eigenfunctions}
 of the ``adjoint" linear fourth-order operator
 $$
 \tex{
 {\bf B}^* \equiv -\D^2 -\frac{1}{4}\, y\cdot \nabla \inB \ren.
   }
   $$
 Indeed,  our blow-up patterns, in this particular case and in the rescaled form,
  will be a solution of the rescaled equation
 $$
 Y_t = \hat{{\bf B}}^* Y +D(Y) \inB \ren \times \re_+,
   $$
 where $D(Y)$ is a quadratic perturbation of the operator $\hat{{\bf B}}^*$ as $Y\rightarrow 0$
and
$$
 \tex{
\hat{{\bf B}}^*= -\D^2 -c \D(\frac{1}{y^2}I)-\frac{1}{4}\, y\cdot
\nabla -\frac{1}{2(p-1)}\, I,
 }
$$
 with $c$ being a  certain constant. As a result, we discuss
{\em three types} of blow-up for the unstable
 C--H equation \ef{ch6}.



\setcounter{equation}{0}
\section{Preliminary discussions of  Cahn--Hilliard equations}
 \label{S2}

In its origins, the Cahn--Hilliard equation was proposed as a
continuum model for the description of the dynamics of pattern
formation in phase transition. When a binary solution is cooled
sufficiently, phase separation may occur and then proceed in two
ways: either nucleation, in which nuclei of the second phase
appear randomly and grow, or, in the so-called spinodal
decomposition, the whole solution appears to nucleate at once and
then periodic or semi-periodic structures appear. Pattern
formation resulting from phase transition has been observed in
alloys, glasses, and polymer solutions. From the mathematical
point of view, this equation involves a fourth order elliptic
operator and it contains a negative viscosity term. The unknown
function is a scalar $u=u(x,t)$, $x\in \ren$, $t\in\re_+$ and the
equation reads
\begin{equation}
\label{ch1}
    u_{t} - \D K(u)=0\quad \hbox{in} \quad \ren \times \re_+\,
    \whereA
    K(u):= -\nu \D u + f(u), \quad \nu >0,
\end{equation}
and the function $f(u)$ is a polynomial of the order $2p-1$,
\begin{equation*}
 \tex{
    f(u):= \sum_{j=1}^{2p-1} a_j u^j, \quad p\in \N,\;\; p\geq 2.
    }
\end{equation*}
In particular, the so-called Cahn--Hilliard equation corresponds
to the case $p=2$ and
\begin{equation*}
    f(u):= -\eta u+ \mu u^3, \quad \eta,\mu >0.
\end{equation*}
In the case we consider the problem in an open bounded domain $\O$ of
$\ren$, with a smooth boundary $\G:=\p \O$, we can suppose the
following boundary conditions. Either Neumann boundary conditions
\begin{equation*}
 \tex{
    \frac{\p u}{\p {\bf n}}= -\frac{\p }{\p {\bf n}}K(u)=0 \quad \hbox{on} \quad \G,
    }
\end{equation*}
where ${\bf n}$ is the unit normal forward to $\G$ (or Dirichlet
boundary conditions). Or, assuming that $\O= \prod_{i=1}^N
(0,L_i)$, $L_i>0$, the periodic boundary condition
\begin{equation*}
    \var|_{x_i =0} = \var|_{x_i = L_i}, \quad i=1,\cdots,N,
\end{equation*}
for $u$ and the derivatives of $u$ at least of order $\leq 3$. The
problem \eqref{ch1} can be completed with the initial-value
conditions
\begin{equation*}
    u(x,0)=u_0(x), \quad x\in \O.
\end{equation*}
The weak formulation of the problem is obtained by multiplying
\eqref{ch1} by  a test function $\var \in \mc{C}_0^\infty(\O)$,
integrating in $\O$ and applying the formula of integration by
parts,
\begin{equation}
\label{ch2}
 \tex{
    \int\limits_\O u_t \var + \int\limits_\O \nu \D u \D\var - \int\limits_\O f(u) \D\var = 0.
    }
\end{equation}
Integrating again by parts yields
\begin{equation}
\label{ch3}
 \tex{
    \int\limits_\O u_t \var + \int\limits_\O \nu \D u \D\var + \int\limits_\O f'(u)
    \nabla u \cdot \nabla\var = 0 \whereA f':=\frac{{\mathrm d}f}{{\mathrm d}u}.
 }
\end{equation}

The dynamical system \eqref{ch1} is {\em gradient}  and admits a
Lyapunov function of the form
\begin{equation}
\label{ch4}
 \tex{
    E[u](t):= \frac{\nu}{2} \int\limits_\O |\nabla u|^2 + \int\limits_\O
    g(u), }
\end{equation}
where $g(u)$ is the primitive of $f(u)$ and assuming that $u$ is a
sufficiently regular solution of the problem. Multiplying
\eqref{ch1} by $K(u)$, integrating in $\O$, and applying the
formula of integration by parts we find that
 $$
\begin{matrix}
 \tex{
    \int\limits_\O K(u) u_t = -\nu \int\limits_\O \D u \cdot u_t + \int\limits_\O f(u) u_t=
    \frac{\mathrm d}{{\mathrm d}t}E[u](t), \quad \mbox{and}
 } \\
   \tex{ -\int\limits_\O \D K(u) K(u) = -\int\limits_\G
    \frac{\p K(u)}{\p \nu} K(u) \,{\mathrm d}\G +\int\limits_\O |\nabla K(u)|^2=
    \int\limits_\O |\nabla K(u)|^2.}
\end{matrix}
 $$
Therefore,
\begin{equation}
\label{ch5}
 \tex{
    \frac{\mathrm d}{{\mathrm d}t}E[u](t)+ \int\limits_\O |\nabla K(u)|^2=0,}
\end{equation}
 so that the Lyapunov function \eqref{ch4} is monotone decreasing
in time. It should be pointed out that these properties can be
accomplished when $\O = \ren$.
\par
Furthermore, we would finally also like to note that when \eqref{11} is reduced to a standard
Cahn--Hilliard equation, the {\em limit} unstable Cahn--Hilliard
equation \ef{ch6} was studied  in \cite{EGW}
connecting this model with various applications. In particular,
if $N=1,2$ and $p=3$, it arises as the limit case of the
phenomenological unstable \emph{Cahn--Hilliard equation}
\begin{equation*}
    u_{t} = -(\g u_{xx} - u^3 +\g_1 u)_{xx}- \g_2 u\,.
\end{equation*}
It is also a reduced model from solidification theory with $N=1$ or 2 and $p=2$
 (see \cite{BB, NC}).
Equations of this form arise in the theory of thermo-capillary
flows in thin layers of viscous fluids with free boundaries and an
anomalous dependence of the surface tension coefficient on
temperature (see \cite{AP, F}). Also, equation \eqref{ch6} occurs
passing to the limit as $\g \rightarrow 0^+$ in the Cahn--Hilliard
equation
\begin{equation}
\label{ch7}
    u_{t} = \nabla \cdot (\nabla (F(u)- \D u))\,,
\end{equation}
with a standard double-well potential function of the form
\begin{equation*}
    F(u)= \left|u\right|^{p-1}u - \g \left|u\right|^{p}u \whereA
    \g>0.
\end{equation*}
\par
Another important class of fourth-order models related to \eqref{ch6} comes from the theory
of thin films and general long-wave unstable equations (see \cite{BB} and \cite{WBB}),
where a typical quasi-linear equation takes the form
\begin{equation*}
    u_{t} = -(u^{n} u_{xxx} +u^{m} u_{x})_{x}\,.
\end{equation*}
Observe that here a linear perturbation of the case  $n=0$ is treated.

\setcounter{equation}{0}
\section{Fibering method for a stationary unstable C--H equation}
 \label{S3}

 In this section, we study the existence and
multiplicity of solutions of the following stationary
unstable {C--H}-type equation:
\begin{equation}
\label{fb1}
    -\Delta^2 u + \g u - \Delta (|u|^{p-1}u)=0 \quad
    \hbox{in}\quad \O \quad (p>1),
\end{equation}
with the Navier boundary conditions as in \ef{12}. We will focus
on achieving such results depending on the value of the parameter
$\g$ and considering the equation in a bounded domain $\O \subset
\ren$ with Navier-type boundary conditions \eqref{12}.

\ssk

\noi{\bf Remark: on setting in $\ren$.} It turned out that, for
$\g <0$, the problem can be posed in the whole of $\ren$ in a
class of functions properly decaying at infinity,
\begin{equation*}
    \tex{\lim_{|x| \rightarrow \infty} u(x)=0;}
\end{equation*}
see a brief discussion in Section \ref{S.RN} below.
 However, there are specific difficulties concerning a suitable
 functional setting of the \ef{fb1} in $\ren$, so this will be
 done in a separate paper \cite{CHPabloRN}, where  a very wide
 class of solutions (critical points of a functional) is detected.


\ssk

Thus, to carry out the study of \ef{fb1} in a bounded smooth $\O
\subset \ren $, we will use the {\em fibering method}, introduced
by S.I.~Pohozaev in the 1970s \cite{Poh0, PohFM}, as a convenient
generalization of previous versions by Clark and Rabinowitz
\cite{Clark, Rabin}  of variational approaches, and further
developed   by Dr\'abek and Pohozaev \cite{DraPoh} and others in
the 1980's. In particular, recently, it was used by Brown and
collaborators \cite{BrownWu1,BrownWu2} to ascertain the existence
and multiplicity of solutions for equations with a variational
form (in particular p-Laplacian) associated to such equation,
i.e., potential operator equations, alternatively to other methods
such as bifurcation theory, critical point theory and so on.


\subsection{Preliminary results for the variational analysis}


Firstly, observe that \eqref{fb1} is not variational in $L^2(\O)$,
though it is variational in $H^{-1}(\O)$. Thus, multiplying
\eqref{fb1} by $(-\D)^{-1}$,
 we obtain a nonlocal elliptic equation with the standard zero
 Dirichlet boundary condition\footnote{Here, as customary,
$(-\D)^{-1}u=v$, if
 $$
 -\D v=u \,\,\, \mbox{in} \,\,\, \O, \quad v=0 \,\,\, \mbox{on}
 \,\,\, \p \O.
 $$
Therefore,  \ef{fbnew}  implies that $\D u=0$ on $\p \O$, so that
both the Navier conditions in \ef{12} hold for $u$.}
 \be
 \label{fbnew}
 \tex{-\D u - \g (-\D)^{-1} u - |u|^{p-1} u=0 \,\,\, \mbox{in} \,\,\, \O,
 \quad u=0 \,\,\, \mbox{on} \,\,\, \p \O,}
 \ee
 which  admits a variational setting in $L^2(\O)$ and, hence, the fibering method can be applied.
 To this end, consider
 the following Euler functional associated to \eqref{fbnew}:
\begin{equation}
\label{funfb}
    \tex{ \mc{F}_\g(u):=\frac{1}{2} \int\limits_\O |\nabla u|^2 - \frac{\g}{2} \int\limits_\O |(-\D)^{-1/2} u|^2 -
     \frac{1}{p+1} \int\limits_\O
    |u|^{p+1},
    }
\end{equation}
 such that the solutions of \eqref{fbnew} can be obtained as
critical  points of the $C^1$ functional \eqref{funfb}. Note that
$(-\D)^{-1}$ is  a positive linear integral compact  operator from
$L^2(\O)$ to itself. Then, the operator $(-\D)^{-1/2}$ is defined
as the square root of the operator $(-\D)^{-1}$ and it will also be
referred to as a non-local compact linear operator.


Subsequently, for the functional \eqref{funfb}, the following
result is  well-known. Hereafter, we are assuming that $H_0^1(\O)=
W_0^{2,1}(\O)$.

\begin{lemma}
\label{Le fb21}
The functional \eqref{funfb} is Fr\'echet differentiable
and its Fr\'echet derivative is
\begin{equation*}
    \tex{D_{u}\mc{F}_\gamma(u) \varphi
    :=\int\limits_\O \nabla u \cdot \nabla \varphi - \g \int\limits_\O (-\D)^{-1/2}
    u\cdot
    (-\Delta)^{-1/2} \varphi  - \int\limits_\O |u|^{p} \varphi,}
     \quad \varphi \in H_0^1(\O).
\end{equation*}


\end{lemma}
\begin{proof}
Let $\mc{F}_\g(u+\varphi)$ be
\begin{equation*}
    \tex{\mc{F}_\g(u+\varphi)
    :=\frac{1}{2}\big[\int\limits_\O |\nabla (u+\varphi)|^2
    - \g \int\limits_\O
    \left|(-\Delta)^{-1/2} (u+\varphi)\right|^2 \big] -
    \frac{1}{p+1}\int\limits_\O |u+\varphi|^{p+1}.}
\end{equation*}


We split the proof between two parts. The first one obtaining the Fr\'echet derivative for the first two terms of the functional, denoted by the
functional $\mc{F}_{1,\g}$, and the second for the non-linear part, denoted by $\mc{F}_{2,\g}$. Subsequently, operating the expressions for the first
two terms of the functional and rearranging terms yields
\begin{align*}
    \tex{\mc{F}_{1,\g}(u+\varphi)=\frac{1}{2}[\int\limits_\O |\nabla u|^2 } & \tex{- \g \int\limits_\O
    \left|(-\Delta)^{-1/2} u\right|^2+ \int\limits_\O |\nabla \varphi|^2
    - \g \int\limits_\O
    \left|(-\Delta)^{-1/2} \varphi\right|^2 ]}  \\ &  \tex{+
    \int\limits_\O \nabla u \cdot \nabla \varphi -\g \int\limits_\O  (-\Delta)^{-1/2} u \cdot
    (-\Delta)^{-1/2} \varphi.}
\end{align*}
Since, $\tex{\int_\O |\nabla \varphi|^2 - \g \int_\O
    \left|(-\Delta)^{-1/2} \varphi\right|^2}$ vanishes quite radically
\begin{equation*}
    \tex{\left|\int\limits_\O |\nabla \varphi|^2 - \g \int\limits_\O
    \left|(-\Delta)^{-1/2} \varphi\right|^2 \right| \leq K
    \left\|\varphi\right\|_{H_0^1(\O)}=
    o(\left\|\varphi\right\|_{H_0^1(\O)}),}
\end{equation*}
as $\left\|\varphi\right\|_{H_0^1(\O)}$ goes to zero and for a
positive constant $K$, we find that $$
    \tex{ \left|\mc{F}_{1,\g}(u+\varphi)-\mc{F}_{1,\g}(u)
    - \int\limits_\O \nabla u \cdot \nabla \varphi  + \g \int\limits_\O  (-\Delta)^{-1/2} u \cdot
    (-\Delta)^{-1/2} \varphi  \right|
    =   o(\left\|\varphi \right\|_{H_0^1(\O)}),}
$$
as $\varphi \rightarrow 0$ in
$H_0^1(\O)$.

Furthermore, for the term related to the nonlinear part (the third
term in the functional \eqref{funfb}),  we use the Taylor's
expansion in $\varphi=0$, such that
\begin{equation*}
 \tex{
    \frac{1}{p+1}|u+\varphi|^{p+1}=\frac{1}{p+1}|u|^{p+1}+ |u|^{p} \varphi+ o(\left|\varphi\right|),
 }
\end{equation*}


\noi as $\varphi\rightarrow 0$. Therefore, since $u\in H_0^1(\O)$,
we can conclude that
\begin{equation*}
 \tex{
    \Big|\frac{1}{p+1}\int\limits_\O |u+\varphi|^{p+1} -\frac{1}{p+1} \int\limits_\O |u|^{p+1}
    -\int\limits_\O |u|^{p}\varphi\Big|
    = o\big(\big\|\varphi\big\|_{H_0^1(\O)}\big),
     }
\end{equation*}
when $\varphi$ goes to zero in $H_0^1(\O)$, which completes the
proof.
\end{proof}

Consequently, we have the directional derivative (Gateaux's
derivative) of the functional \eqref{funfb} as follows:
\begin{equation}
\label{derv}
    \tex{\frac{\mathrm d}{{\mathrm d}t}\, \mc{F}_\g(u+t\varphi)_{|t=0}= \left\langle \varphi, D_u \mc{F}_\g(u)\right\rangle
    =   D_{u}\mc{F}_\g(u) \varphi.}
\end{equation}
Furthermore, due to \eqref{derv}, the critical points of
\eqref{funfb} are weak solutions in $H_0^1(\O)$ for the equation
\eqref{fbnew}. In other words, the Fr\'echet derivative obtained
in Lemma\;\ref{Le fb21} of the functional \eqref{funfb} is going
to be zero when $u$ is a weak solution of \eqref{fbnew}, i.e.,
\begin{equation}
\label{deriva}
    \tex{D_{u}\mc{F}_\g(u) \varphi=0.}
\end{equation}

 We denote  critical points of the functional $\mc{F}_\g(u)$
\eqref{funfb} as follows:
\begin{equation*}
    \mc{C}_\g:=\{u \in W_0^{2,1}\,:\,
    D_{u}\mc{F}_\g(u) \varphi=0\}.
\end{equation*}




 Then, as usual, the critical
points of the functional $\mc{F}_\g(u)$ \eqref{funfb} correspond
to weak solutions of the equation \eqref{fbnew} and, hence, to the stationary Cahn--Hilliard equation
\eqref{fb1}, i.e.,
\begin{equation}
 \label{defweak11}
    \tex{\int\limits_\O \nabla u \cdot \nabla \varphi - \g \int\limits_\O (-\D)^{-1/2} u \cdot
    (-\Delta)^{-1/2} \varphi  - \int\limits_\O |u|^{p}  \varphi=0,  }
\end{equation}
 for any $\varphi \in W_0^{2,1}(\O)$ (or $C_0^\infty(\O)$). Thus,
$u\in \mc{C}_\g$ if and only if
\begin{equation}
\label{critp}
    \tex{\int\limits_\O |\nabla u|^2 - \g \int\limits_\O \left|(-\D)^{-1/2} u\right|^2 - \int\limits_\O
    |u|^{p+1}=0.  }
\end{equation}



By classic elliptic regularity for higher-order equations
(Schauder's theory; see \cite{Berger} for further details), we
will then always obtain classical solutions for such equations.

For the sake of completion, we study some of the properties of
the functional $\mc{F}_\g(u)$ in \eqref{funfb}.
To do so, the following definitions are convenient
to introduce:


\begin{definition}
\label{Defb1} Given a  map $\mc{F}:V\longrightarrow \ren$, where
$V$ is a Banach space, it is weakly (sequentially) lower
semicontinuous $(${\sc wls}$)$, if, for any weakly convergent
sequence $\{u_n\}$ in $V$, $u_n \rightharpoonup u$, as
$n\rightarrow \infty$, there holds
\begin{equation*}
    \mc{F}(u)\leq \liminf_{n\rightarrow \infty}\mc{F}(u_n).
\end{equation*}
\end{definition}

\begin{definition}
\label{Defb2} Given a map $\mc{F}:V\longrightarrow \ren$, where
$V$ is a Banach space, it is weakly semicontinuous $(${\sc ws}$)$,
if, for any weakly convergent sequence $\{u_n\}$ in $V$, $u_n
\rightharpoonup u$, as $n\rightarrow \infty$, there holds
\begin{equation*}
    \mc{F}(u)= \lim_{n\rightarrow \infty}\mc{F}(u_n).
\end{equation*}
\end{definition}

Next,  the following lemma is easily proved:
\begin{lemma}
\label{lem1} If X is a Hilbert space, then its norm is {\sc wls}.
\end{lemma}

\begin{proof}
Since the square root function is a continuous function, we find
that
\begin{equation*}
    \tex{\left\|u\right\|_{X}^2 \leq \liminf
    \left\|u_n\right\|_{X}^2 \quad \LongA \quad
    \left\|u\right\|_{X} \leq \liminf
    \left\|u_n\right\|_{X},}
\end{equation*}
for any sequence $\{u_n\}$ in the space $X$ convergent to $u\in
X$. Thus, firstly, we assume that $u_n \rightharpoonup u$ in X and
by definition we also have that
\begin{equation*}
    \tex{0\leq\left\|u_n-u\right\|_{X}^2
    =\left\|u_n\right\|_{X}^2-
    2\left\langle u_n,u\right\rangle_{X}
    +\left\|u\right\|_{X}^2,}
\end{equation*}
where, $\left\langle \cdot,\cdot\right\rangle_{X}$ represents the
inner product of the Hilbert space $X$. Hence,
\begin{equation}
\label{fb2}
    \tex{2\left\langle u_n,u\right\rangle_{X}
    -\left\|u\right\|_{X}^2
    \leq \left\|u_n\right\|_{X}^2.}
\end{equation}
Moreover, owing to the convergence of the taken sequence, we can
choose a subsequence of $\left\|u_n\right\|_{X}^2$, convergent to
$\liminf \left\|u_n\right\|_{X}^2$. Therefore, passing to the
limit \eqref{fb2}, we find that
\begin{equation*}
    \tex{\left\|u\right\|_{X}^2 \leq \liminf
    \left\|u_n\right\|_{X}^2,}
\end{equation*}
which concludes the proof.
\end{proof}

Then, assuming that $u\in W_0^{2,1}(\O)$ is equipped with the
norm
\begin{equation*}
    \tex{\left\|u\right\|_{W_0^{2,1}(\O)} := \big(\int\limits_\O |\nabla u|^2\big)^{1/2}}
\end{equation*}
(this is possible thanks to Poincar\'e's inequality) and applying
Lemma\,\ref{lem1} it is clear that the functional
\begin{equation*}
    \tex{u \rightarrow \int\limits_\O |\nabla u|^2,}
\end{equation*}
is weakly lower semicontinuous.





Furthermore, it is also easy to prove that the second and third term of the
functional \eqref{funfb} are weakly semicontinuous.

\begin{lemma}
\label{lem2}
Suppose $u\in H_0^1(\O)$. Then, $\int\limits_\O
    \left|(-\Delta)^{-1/2} u\right|^{2}$ is {\sc ws}.
\end{lemma}
\begin{proof}


As performed in the proof of Lemma\;\ref{lem1}, we take a
convergent sequence $\{u_n\}$ in $H_0^1(\O)$ so that $u_n
\rightharpoonup u$ for some $u\in H_0^1(\O)$. Then, $(-\D)^{-1/2}
u_n :=f_n$, with $\{f_n\} \subset H_0^{1/2}(\O)$ being
equicontinuous in $H_0^{1/2}(\O)$. Then, by the compact imbedding
of $H_0^{1/2}(\O)$ into $L^2(\O)$ and by the Ascoli--Arzel\'a
theorem we can extract a convergent subsequence $\{f_{m_i}\}$ in
$L^2(\O)$  so that $f_{m_i}\rightarrow f$ as $m_i
\rightarrow\infty$. Moreover, since the linear operator
$(-\D)^{-1/2}$ is compact,  we find that
\begin{align*}
    \tex{f_{m_i}\rightarrow f} & \tex{
    \LongA (-\Delta)^{-1/2} u_{m_i} \rightarrow
    (-\Delta)^{-1/2} u} \\  & \tex{\LongA
    \int\limits_\O \left|(-\Delta)^{-1/2} u_{m_i}\right|^{2}
    \rightarrow \int\limits_\O \left|(-\Delta)^{-1/2} u\right|^{2}.}
\end{align*}
This completes the proof.
\end{proof}


\begin{lemma}
\label{lem3} Suppose $u\in H_0^{1}(\O)=W_0^{2,1}(\O)$. Then,
$\int\limits_\O
    |u|^{p+1}$ is {\sc ws}, if $p<\frac{N+2}{N-2}$.
\end{lemma}


\begin{proof}

As demonstrated in the proof of Lemma\;\ref{lem1}, we take a
convergent sequence $\{u_n\}$ in $W_0^{2,1}(\O)$ so that $u_n
\rightharpoonup u$ for some $u\in W_0^{2,1}(\O)$. Then, standard
functional analysis tells us that the functions $u_n$ are bounded
in $W_0^{2,1}(\O)$ and in $L^\infty(\O)$ by Sobolev's inequality
when $N=1,2$. Therefore, $\{u_n\}$ satisfies the Ascoli--Arzel\'a
theorem, so we can extract a convergent subsequence so that
$u_{n_i}\rightarrow u$ as $n_i \rightarrow\infty$ in
$L^\infty(\O)$. Thus,
\begin{equation*}
    \tex{u_{n_i}\rightarrow u
    \LongA \int\limits_\O |u_{n_i}|^{p+1} \rightarrow
    \int\limits_\O |u|^{p+1} .}
\end{equation*}


Furthermore, when $N>2$, by Sobolev's inequality, we have the imbedding of
$W_0^{2,1}(\O)$ into $L^q(\O)$, with $q=\frac{2N}{N-2}$.
 Then, if $p<\frac{N+2}{N-2}(=p_{S})$ the imbedding of
$W_0^{2,1}(\O)$ into $L^{p+1}(\O)$ is compact and by the continuity of the
Nemytskii operator $f(x,u):=|u|^{p+1}$ we can extract again a
convergent subsequence that proves the weakly semicontinuity for
this particular case.


\end{proof}

Indeed, we observe that by Fatou's Lemma and the continuity of the Nemytskii
operator $f(x,u):=|u|^{p+1}$ it is possible to find a convergent
subsequence $\{u_{n_i}\}$ such that, for any $p>1$
$$\tex{f(x,u)=|u|^{p+1}\leq \liminf_{n_i \rightarrow \infty}
|u_{n_i}|^{p+1},} \quad \mbox{and}$$  $$ \tex{\int\limits_\O
|u|^{p+1} \leq \liminf_{n\rightarrow \infty}  \int\limits_\O
|u_{n_i}|^{p+1}.
 }
 $$









\subsection{Direct application of the fibering method}


\noindent Subsequently, in order to apply the fibering method, we
split the function $u\in W_0^{2,2}(\O)$ as follows (without loss
of generality, we can suppose that $u\in W_0^{2,1}(\O)$):
\begin{equation}
\label{split}
    \tex{ u(x)=r v(x),}
\end{equation}
 where $r\in \re$, such that $r\geq 0$, and $v\in W_0^{2,1}(\O)$,
 to obtain the so-called {\em fibering maps}
\begin{align*}
    \tex{\phi_v\,:\,} & \tex{ \re \rightarrow \re,}\\  &
    \tex{r \rightarrow \mc{F}_\g(rv).}
\end{align*}
Substituting  $u$ from \ef{split} into  the functional
\eqref{funfb}, we have that
\begin{equation}
\label{funsplit}
    \tex{ \phi_v(r)= \mc{F}_\g(rv):=\frac{r^2}{2} \int\limits_\O
    |\nabla v|^2 - \frac{r^2\g}{2} \int\limits_\O |(-\D)^{-1/2} v|^2 -
    \frac{r^{p+1}}{p+1} \int\limits_\O
    |v|^{p+1}.}
\end{equation}
Thus,  \eqref{funsplit} defines the current fibering maps.


 Note that, if $u\in W_0^{2,1}(\O)$ is a critical point of
$\mc{F}_\g(u)$, then
\begin{equation*}
 \tex{
 D_u \mc{F}_\g(rv)v= \frac{\partial \mc{F}_\g(rv)}{\partial r}=0.
  }
\end{equation*}
In other words, $D_u \mc{F}_\g(rv)v=\left\langle D_u
\mc{F}_\g(rv),v\right\rangle_{W_0^{2,1}(\O)}$. Here, we denote by
$\left\langle\cdot,\cdot\right\rangle_{W_0^{2,1}(\O)}$ the inner
product in the space $W_0^{2,1}(\O)$. Thus, the calculation of
that derivative yields
\begin{equation*}
    \tex{ \phi_v'(r)=r\int\limits_\O |\nabla v|^2 -
    r\g \int\limits_\O |(-\D)^{-1/2} v|^2 - r^{p}\int\limits_\O
    |v|^{p+1} .}
\end{equation*}


\noi Moreover, since we are looking for non-trivial solutions
(critical points), i.e., $u\neq 0$,  we have to assume that $r\neq
0$. Hence,
\begin{equation}
\label{varivu}
    \tex{ \int\limits_\O |\nabla v|^2 - \g \int\limits_\O |(-\D)^{-1/2} v|^2 -
     r^{p-1}\int\limits_\O
    |v|^{p+1} =0,}
\end{equation}
and assuming that $\int\limits_\O |v|^{p+1} \neq 0$, we finally arrive at
\begin{equation}
\label{rex}
    \tex{r^{p-1} = \frac{\int\limits_\O |\nabla v|^2 -
     \g \int\limits_\O |(-\D)^{-1/2} v|^2 }{\int\limits_\O |v|^{p+1} }>0.}
\end{equation}



Now, calculating $r$ from \eqref{rex} (values of the scalar
functional $r=r(v)$, where those critical points are reached) and
substituting it into \eqref{funsplit} gives the following
functional:
\begin{equation}
\label{spfunc}
    \tex{ \mc{G}_\g(v)=\mc{F}_\g(r(v)v):=\big(\frac{1}{2}- \frac{1}{p+1}\big)
    \frac{\big( \int\limits_\O |\nabla v|^2 - \g \int\limits_\O |(-\D)^{-1/2} v|^2\big)^{\frac{p+1}{p-1}}}
    {\big(\int\limits_\O |v|^{p+1}  \big)^{\frac{2}{p-1}}}\,}.
\end{equation}



\noi According to Dr\'abek--Pohozaev \cite{DraPoh}, $r=r(v)$ is
well-defined and consequently the fibering map \eqref{funsplit}
possesses a unique point of monotonicity change in the case
\be
\label{poss} \tex{
\int\limits_\O |\nabla v|^2 - \g
\int\limits_\O |(-\D)^{-1/2} v|^2>0 \quad \hbox{and}\quad
\int\limits_\O |v|^{p+1}>0.} \ee
This is explained in detail  later on, when the analysis of the
fibering maps is carried out.


Furthermore, thanks to \cite[Lemma\;3.2]{DraPoh}, we can assume
that the Gateaux derivative of the functional $\mc{G}_\g$ at the
point $v\in W_0^{2,1}(\O)$ in the direction of $v$ is zero, i.e.,
 $$
 \left\langle D_v \mc{G}_\g(v),v\right\rangle_{W_0^{2,1}(\O)}=0.
  $$







\noi Therefore, assuming that $v_c$ is a critical point of
$\mc{G}_\g$, by the transformation carried out above, we have that
a critical point $u_c \in W_0^{2,1}(\O)$, $u_c \neq 0$, of
$\mc{F}_\g$ is generated by $v_c$ through the expression
 $$
 u_c=r_c
v_c,
$$
 with $r_c$ defined by \eqref{rex}.



\subsection{Multiplicity results}


\noindent In the following, we shall provide a description of the
fibering maps associated with \eqref{fb1}. It will become clear
that the essential nature of those fibering maps is determined by
the sign of the terms $\g \int_\O |(-\D)^{-1/2} v|^2$ and $\int_\O
|v|^{p+1}$. Since  $\int_\O |v|^{p+1}$ is always non-negative, the
different possibilities will depend on the value of the parameter
$\g$. Moreover, the different zeros of those fibering maps will
provide us with the critical points of the functional $\mc{G}_\g$
in \eqref{spfunc}, and, hence, by construction, of the functional
$\mc{F}_\g$ given by \eqref{funfb}. This is also supported by the
category analysis of the functional
 $\mc{F}_\g$,
\eqref{funfb}.

Now, following  \cite{BrownWu1,BrownWu2}, we define the function
$\omega_v\,:\,\re_+ \rightarrow \re$ by
\begin{equation}
\label{mult01}
    \tex{ \o_v(r)= \int\limits_\O |\nabla v|^2 - r^{p-1}\int\limits_\O
    |v|^{p+1}.}
\end{equation}
By the definition of the fibering maps and their relations  with
the critical points of the functional \eqref{funfb},
for $r>0$,
$$rv \in \mc{C}_\g \quad \hbox{if and only if}\quad r\; \hbox{is a solution of \eqref{mult01}},$$
\begin{equation}
\label{imp}
    \tex{\hbox{and}\quad \o_v(r)= \g \int\limits_\O |(-\D)^{-1/2} v|^2.}
\end{equation}
Moreover,
\begin{equation*}
    \tex{ \o_v'(r)= -(p-1)r^{p-2}\int\limits_\O
    |v|^{p+1},}
\end{equation*}


\noi  and, hence, $\o_v(r)$ is strictly decreasing, for any $r\geq
0$, since $\int\limits_\O |v|^{p+1}>0$. Also, we have that
\begin{equation}
\label{convex}
    \tex{ \phi_v''(r)=\int\limits_\O |\nabla v|^2 - \g \int\limits_\O |(-\D)^{-1/2} v|^2 - pr^{p-1}\int\limits_\O
    |v|^{p+1},}
\end{equation}
which can
provide us with the convexity of the fibering maps depending on
the increasing or decreasing  function \eqref{imp}. Indeed, if $rv
\in \mc{C}_\g$, i.e., $u$ is a critical point of the
functional \eqref{funfb} that satisfies \eqref{critp}, by
\eqref{varivu}, we have that
\begin{align*}
    \tex{ \phi_v''(r)} & \tex{ = \int\limits_\O |\nabla v|^2 - \g \int\limits_\O |(-\D)^{-1/2} v|^2 -r^{p-1}\int\limits_\O
    |v|^{p+1}- (p-1)r^{p-1}\int\limits_\O
    |v|^{p+1}} \\ & \tex{ = -(p-1)r^{p-1}\int\limits_\O
    |v|^{p+1} ,}
\end{align*}
such that
\begin{equation*}
    \tex{  r^{-1}\o_v'(r)=   \phi_v''(r),}
\end{equation*}
and then, we can say that the fibering map $\phi_v$ is always concave.


\noi Furthermore, we define the function
\begin{equation*}
    \tex{\mc{H}_v(r):=\frac{r^2}{2} \int\limits_\O |\nabla v|^2 - \frac{r^{p+1}}{p+1} \int\limits_\O
    |v|^{p+1}.}
\end{equation*}
Observe that, if the parameter $\g$ is less than a certain value
(to be specified in detail below), we find that the fibering map
is positive, $\phi_v(r)>0$, when $\int_\O |v|^{p+1}\geq 0$ (note
that
the opposite inequality $\int_\O |v|^{p+1}< 0$ is not possible) up to a critical value of r, i.e., for sufficiently small r's.
Then, the functional $\mc{H}_v(r)$ has a unique critical point at
the value $r=r_{\max}$ such that
\begin{equation*}
    \tex{r_{\max}= \Big(\frac{\int_\O |\nabla v|^2}{\int_\O
    |v|^{p+1}}\Big)^{\frac{1}{p-1}},}
\end{equation*}
and $\mc{H}_v(r)$ takes that maximum value at
\begin{equation*}
    \tex{\mc{H}_v(r_{\max})= \big(\frac{1}{2}-\frac{1}{p+1}\big)\Big(\frac{\big(\int_\O
    |\nabla v|^2\big)^{p+1}}{\big(\int_\O
    |v|^{p+1}\big)^2}\Big)^{\frac{1}{p-1}}.}
\end{equation*}
Note that $\mc{H}_v(r)$ is clearly increasing in the interval
$(0,r_{\max})$, for sufficiently small r's. Subsequently, by the Sobolev compact imbedding of
$W_0^{2,1}(\O)$ into $L^{p+1}(\O)$, with $1<p<\frac{N+2}{N-2}$ if
$N>2$ and any $p>1$ if $N=1,2$, we have that
\begin{equation*}
    \tex{\mc{H}_v(r_{\max})\geq \big(\frac{1}{2}-\frac{1}{p+1}\big)\big(\frac{1}{K_1}\big)
    ^{\frac{1}{p-1}},}
\end{equation*}
where $K_1>0$ is the constant of such imbedding. Besides, we obtain
the following inequality:
\begin{equation}
\label{ineq1}
    \tex{\frac{r_{\max}^2}{2} \int\limits_\O |(-\D)^{-1/2}v|^2 \leq \frac{K_2}{2}
    \Big(\frac{\int_\O |\nabla v|^2}{\int_\O
    |v|^{p+1}} \Big)^{\frac{2}{p-1}} \int\limits_\O |\nabla v|^2=
    \frac{K_2}{2} \Big(\frac{\big(\int_\O |\nabla v|^2\big)^{p+1}}{\big(\int_\O
    |v|^{p+1}\big)^2}\Big)^{\frac{1}{p-1}},}
\end{equation}
where $K_2$ is the corresponding constant for the imbedding of $H_0^{1}(\O)$ into $H_0^{1/2}(\O)$. Hence,
\begin{equation*}
    \tex{\frac{r_{\max}^2}{2} \int\limits_\O v^2 \leq K_2 \frac{p+1}{p+3} \mc{H}_v(r_{\max})= M \mc{H}_v(r_{\max}),}
\end{equation*}
for some constant $M=K_2 \frac{p+1}{p+3}>0$ independent of $v$. Thus,
\begin{equation*}
    \tex{\phi_v(r_{\max})\geq \mc{H}_v(r_{\max}) - \g M \mc{H}_v(r_{\max})= \mc{H}_v(r_{\max}) (1-\g M),}
\end{equation*}
and, hence, $\phi_v(r_{\max})>0$ for all non-zero $u$ if $\g < \frac{1}{M}$, providing a critical value of the
parameter in obtaining the different possibilities for the existence and multiplicity of solutions, i.e., critical
points, for the functional \eqref{funfb}. Note that the constant $M$ might be
equivalently obtained using the expression of the first eigenvalue
of the problem
\begin{equation}
\label{eigp}
    \tex{\Delta^2 u=\lambda u,}
\end{equation}
under homogeneous Dirichlet boundary conditions, i.e.,
\be
\label{dbc}
u=0,\quad \nabla u=0 \quad \hbox{on}\quad \p\Omega,
\ee
or Navier-type boundary conditions \ef{12} imposed for the problem \eqref{fb1}.
Thus, a real number $\l$ is called an eigenvalue and $u\in
W_0^{2,2}(\O)$, $u\neq0$, its corresponding eigenfunction if
\begin{equation*}
    \tex{\int\limits_\O \Delta u \, \Delta \varphi = \l \int\limits_\O u
    \varphi \quad \hbox{for any}\quad \varphi \in W_0^{2,2}(\O).}
\end{equation*}
Indeed, let us observe that the first eigenvalue $\l_1$ is positive by definition after integration by parts
\begin{equation*}
    \tex{\l_1:= \min_{u \in W_0^{2,2}(\O)} \frac{\int_\O |\Delta u|^2}{\int_\O u^2}>0.}
\end{equation*}
and, in addition,
for harmonic operators it is well known that the first eigenfunction is always positive too. In particular, for the Laplacian
  $(-\D) >0$, this is  Jentzsch's classic theorem (1912) on the positivity
 of the first eigenfunction for linear integral operators with positive
 kernels (a predecessor of the Krein--Rutman theorem). However,
for poly-harmonic operators $(-\D)^m$, with $m>1$, the first eigenfunction $\phi_1$, associated with the eigenvalue $\l_1$, is not always positive, or
even unique for general domains under Dirichlet boundary conditions of the form \eqref{dbc}.
Both uniqueness and positivity are lost in annuli with very small inner radius (see
\cite{HcS} for further details and discussions).
Therefore, as far as we know, apart from the particular case
when the domain $\O$ is a ball the positivity of the eigenfunction for poly-harmonic operators is still an open problem, even when $\O$ is a smooth
domain. Hence, in general the poly-harmonic operator $(-\D)^m$ in the unit ball, i.e.,  $\O=B_1$, is the only one with sign preserving solutions
 for the Dirichlet problem. In other words, the Green function
of the  poly-harmonic operator  $(-\D)^m$ in the unit ball,
with Dirichlet boundary conditions is known to be positive; see
first results by Boggio (1901-05) \cite{Boggio1, Boggio2} (see
also Elias \cite{Elias} for more recent related general results and
 Grunau--Sweers \cite{HcS}). Moreover, even for nice domains such as an ellipse the solutions might change sign.
 Only certain perturbations of the operator will preserve the sign of the solutions. Performing similar perturbations over the
 domain does not keep the positivity of the solutions either. Apart from partial results in the 2-dimensional case, under some restrictions (see \cite{HcS}), the problem
 in higher dimensions still remains open as well as the situation for general higher order operators.

 On the other hand, we note that the Dirichlet boundary conditions \eqref{dbc} do not allow us to write the eigenvalue problem \eqref{eigp} as a system
 of second order elliptic equations. For Navier boundary conditions as in
 \ef{12}, the first eigenfunction for the problem \eqref{eigp} is always positive by the
 Maximum Principle since in this particular case we can write the problem as a second order elliptic system.

 For convenience, though the known results are not fully classified, we summarize
 the above discussion as follows:



\begin{lemma}
\label{lemmeig}
Let $\l_1$ be the lowest eigenvalue of the problem \eqref{eigp},
 under homogeneous Dirichlet boundary conditions \eqref{dbc}, or Navier-type boundary conditions \eqref{12}
characterized as the minimum of the Rayleigh quotient,
\begin{equation}
\label{ray}
    \tex{\l_1:= \min_{u \in W_0^{2,2}(\O)} \frac{\int_\O |\Delta u|^2}{\int_\O u^2}>0.}
\end{equation}
Moreover, $\l_1$ is algebraically simple and it possesses an
associated eigenfunction denoted
by $\psi_1$. Furthermore, for Navier-type boundary conditions \eqref{12}, the eigenfunction $\psi_1$ is always strictly positive and unique
(up to a multiplicative constant).

In addition, $\l_1$ is the unique and isolated eigenvalue of
\eqref{eigp},
 and any other
eigenvalue $\l_k$, with $k \ge 2$ of \eqref{eigp} satisfies
$\l_k>\l_1$ (there is no eigenvalue less than $\l_1$ and in some
right hand side reduced neighbourhood of $\l_1$ sufficiently small).
Indeed, since the resolvent of the bi-harmonic operator $\Delta^2$ is a compact linear operator in $W_0^{2,2}(\Omega)$
then, the spectrum is discrete, i.e., it might contain either infinitely many isolated
eigenvalues or a finite number of isolated eigenvalues.
\end{lemma}

\begin{remark}
{\rm Note that when the operator is non-self-adjoint there are infinitely many eigenvalues.
Moreover, when we have a non-self-adjoint operator it should be pointed out that the eigenvalues
might be complex apart from the first one, which might be also positive. Then, the dominance of the first eigenvalue
would be represented by
$$\hbox{Re}\;\tau>\l_1,\quad \hbox{for any other eigenvalue $\tau$}.$$
However, since the bi-harmonic operator $\Delta^2$ is self-adjoint all the eigenvalues are real and
the geometric multiplicity equals the algebraic multiplicity.
Moreover, the resolvent of the bi-harmonic operator $\Delta^2$ is a compact so, owing to \cite[Theorem V I.8] {Bre}, the spectrum is discrete.}

\end{remark}



\noi Subsequently, by the expression for the first eigenvalue $\l_1$ of
the problem \eqref{eigp}, we find an equivalent inequality to \eqref{ineq1}
\begin{equation}
\label{ineq2}
    \tex{\frac{r_{\max}^2}{2} \int\limits_\O |(-\D)^{-1/2}v|^2 \leq \frac{1}{2\l_1}
    \Big(\frac{\int_\O |\nabla v|^2}{\int_\O
    |v|^{p+1}}\Big)^{\frac{2}{p-1}} \int\limits_\O |\nabla v|^2=
    \frac{1}{2\l_1} \Big(\frac{\big(\int_\O |\nabla v|^2\big)^{p+1}}{\big(\int_\O
    |v|^{p+1}\big)^2}\Big)^{\frac{1}{p-1}},}
\end{equation}
such that
\begin{equation*}
    \tex{\frac{r_{\max}^2}{2} \int\limits_\O |(-\D)^{-1/2}v|^2 \leq \frac{1}{\l_1}
    \frac{p+1}{p} \mc{H}_v(r_{\max})= M_1 \mc{H}_v(r_{\max}),}
\end{equation*}
with the constant $M_1=\frac{1}{\l_1} \frac{p+1}{p}>0$
independent of $v$ but depending on the first eigenvalue $\l_1$ of
the problem \eqref{eigp}. Thus,
\begin{equation*}
    \tex{\phi_v(r_{\max})\geq \mc{H}_v(r_{\max}) - \g M_1 \mc{H}_v(r_{\max})= \mc{H}_v(r_{\max}) (1-\g M_1),}
\end{equation*}
and, hence, $\phi_v(r_{\max})>0$, i.e.,
$$ \tex{
   \mc{H}_v(r_{\max}) -  \g \frac{r_{\max}^2}{2} \int_\O |(-\D)^{-1/2}v|^2>0,}
   $$
 for all non-zero $u$ provided
that $\g < \frac{1}{M_1}= \l_1 \frac{p}{p+1} $, and
   $$
    \tex{
   \mc{H}_v(r_{\max}) > \big(\frac{1}{2}-\frac{1}{p+1}\big)\big(\frac{1}{K_1}\big)^{\frac{1}{p-1}}.
 }
   $$
These estimations for the parameter $\g$ provide us with a critical value from which we are able to obtain the existence and multiplicity of
solutions for the equation \eqref{fbnew}.
We now discuss the different possibilities
depending
on the possible choices of $\g$, from the previous estimations.
Firstly, note again that, $\int_\O |v|^{p+1}>0$
is always positive. Hence, our analysis will focus on the value of the parameter $\g$.


Thus, if we assume that $\g < \frac{1}{M_1}$, then, since in
this case $$
 \tex{
 \mc{H}_v(r_{\max}) > \big(\frac{1}{2}-\frac{1}{p+1}\big)\big(\frac{1}{K_1}\big)^{\frac{1}{p-1}},
 }
 $$
 it is clear that, by the fibering method, there exists exactly one solution of
\eqref{imp}. Indeed, by the definition and the analysis performed
above, the fibering map $\phi_v(r)$ is a strictly increasing
function for $r<r_{\max}$, and decreasing for $r>r_{\max}$. Thus,
there exists a unique value of $$
r_1(v)>0\quad \hbox{such
that}\quad r_1(v)v=u
  $$
  is a critical point of the functional
$\mc{F}_\g(u)$ in \eqref{funfb}. Also, because $\o_\g'(r_1(v))<0$,
the unique critical point $r=r_1(v)$, that fibering map $\phi_v$
has, will be a local maximum, since $\phi_v''(r_1(v))<0$. In
addition, we have that $\lim_{r\rightarrow +\infty}
\phi_v=-\infty$. This kind of behaviour is shown in Figure
\ref{FF2}.

We must point out, as we shall see below, that, after using
Lusternik--Schnirel'man theory, we cannot assure that there exists
a unique solution since this topological method provides us with a
countable family of solutions and from the fact that the domain
could be very large, the possibility of having more than one
solution cannot be ruled out. Hence, to be precise, we shall say
that there exists at least one solution.

\begin{figure}[!htb]
\begin{center}
\includegraphics[width=10cm]{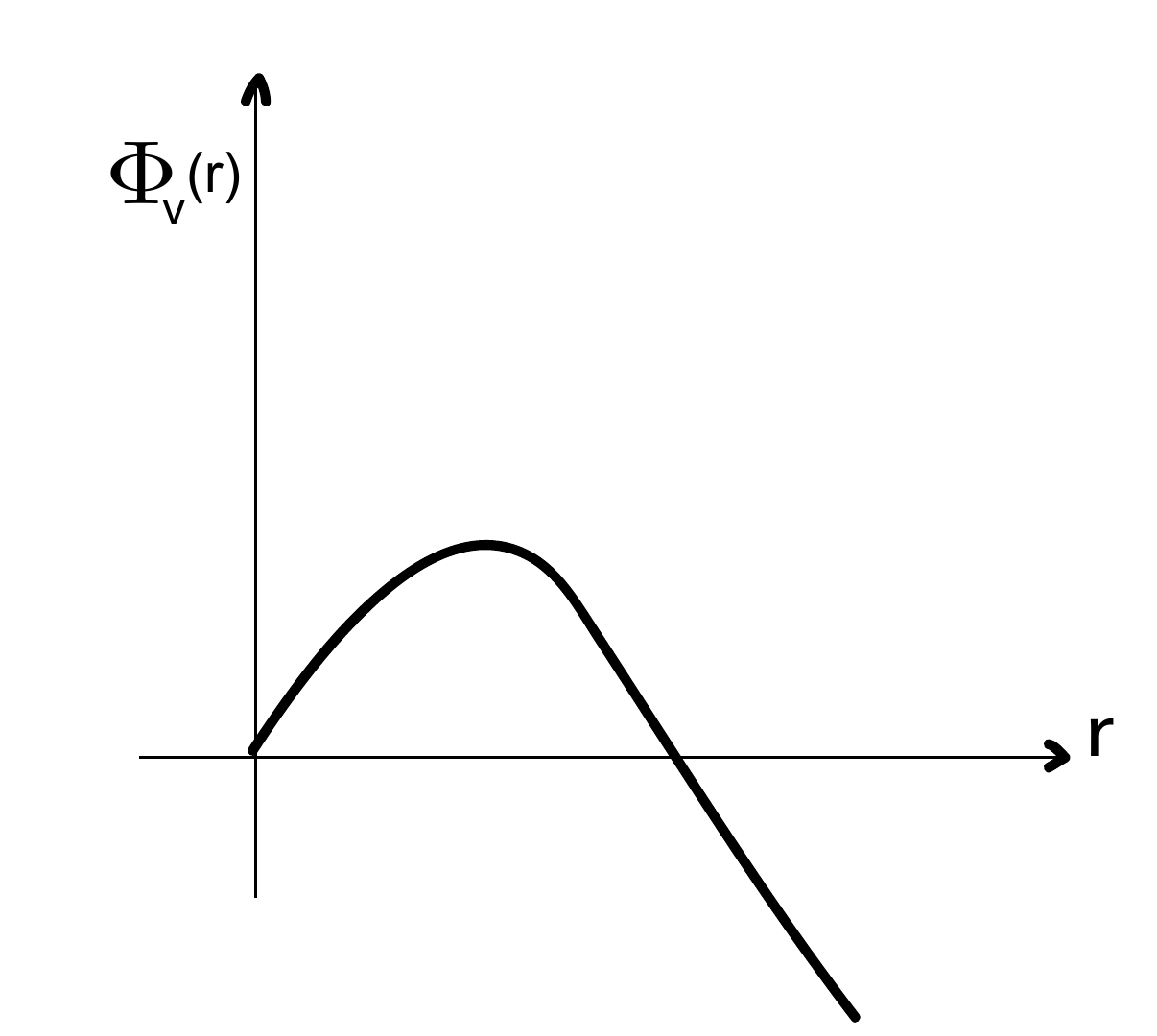}
\caption{Profile of the fibering map for $\g< \frac{1}{M_1}$: a
unique solution.}
 \label{FF2}
\end{center}
\end{figure}

Moreover,  if $\g$ sufficiently large, i.e., $\g> \frac{1}{M_1}$,
and assuming only positive solutions, there is no such critical
points, since the fibering map $\phi_v$ is then  a strictly
decreasing function. However, for oscillatory solutions of
changing sign, we shall show that the number of  possible critical
points of the functional \eqref{funfb} increases with the value of
the parameter $\g$. Indeed, fix a value of the parameter $\g$
bigger than $\frac{1}{M_1}$ but smaller than $\l_\b -\e$, where
$\l_\b$ is the $\b$-eigenvalue of the linear bi-harmonic operator
\eqref{eigp} such that
 \be
 \label{psi11}
 \tex{
 \psi_\b:=\sum_{\left| k\right|=\b} c_k
\hat{\psi}_k,
 }
 \ee
 where  $|\b| > 1$, under the natural ``normalizing"
constraint $$\tex{\sum_{\left| k \right|=\b} c_k=1.}
  $$
   Here, \ef{psi11} represent
the associated eigenfunctions to the eigenvalue $\l_\b$ and
$\{\hat{\psi}_1,\cdots,\hat{\psi}_{M_\b}\}$  is a basis of the
eigenspace of dimension $M_\b$.
 Thus, we
obtain that, for a solution of the form $u=r \psi_\b$, we will
have $M_\b$ corresponding solutions similar to the one obtained in
the previous case, i.e., when the parameter $\g < \frac{1}{M_1}$
represented by \ref{FF2}. Indeed, substituting $u=r \psi_\b$ into
the functional \eqref{funfb} for $\g=\l_\b-\e$, we have
 $$
\tex{\mc{F}_\g(r\psi_\b):=\frac{r^2\e}{2} \int\limits_\O
|(-\D)^{-1/2} \psi|^2 -
    \frac{r^{p+1}}{p+1} \int\limits_\O
    |\psi|^{p+1},}
    $$
  and performing a similar analysis as the one done previously,
   we will have $\b$--critical points (corresponding to the dimension of the eigenspace)
  of the above form represented by Figure \ref{FF2}.
Also, we will provide details below proving that the  number of
solutions can even go to infinity as the parameter $\g$ goes to
infinity.

Nevertheless, to complete the problem, we add a topological analysis to
the algebraic argument mentioned above. In order to estimate the
number of critical points of a functional, we shall need to apply
 Lusternik--Schnirel'man's (L--S) classic theory of calculus of
variations. Thus, the number of critical points of the functional
\eqref{funfb} will also depend on the category of the functional
subset on which the fibering method is taking place.


This topological theory for potential compact operators is a
natural extension of the standard minimax principles which
characterize the eigenvalues of linear compact self-adjoint
operators. Namely, denoting by $\l_1,\l_2,\cdots$  the real
eigenvalues of a self-adjoint compact operator $L$, ordered by
their values with multiplicities, there holds:
$$
  \tex{\l_\b=\sup\limits_{[S^{N-1}]}\,\,\, \min\limits_{v\in S^{N-1}} \left\langle
Lv,v\right\rangle,}
$$
 where $S^{N-1}$ denotes the unit sphere in
an arbitrary $N$-dimensional linear subspace $\Sigma$ of the
corresponding functional space $H$, and $[S^{N-1}]$ denotes the
class of such spheres as $\Sigma$ varies in $H$. Thus, applying
the calculus of variations theory to an operator $L$, the
eigenvalues of the operator $L$ are precisely the critical values
of the functional $\left\langle Lv,v\right\rangle$ on the unit
ball $\p\Sigma=\{v\,:\,||v||=1\}$ of $H$.


The question was how/if that idea (involving eigenvalues) could be
extended to more general {\em nonlinear potential} operators and,
hence, general smooth functionals. To do so,
Lusternik--Schnirel'man introduced the concept of {\em category}
providing an estimate of the number of different critical points
of a functional on the projective spaces. However, the first
problem to be faced is that one needs to find the corresponding
and suitable functional subsets. Introducing the topological
concept of the {\em genus} of a set (that we will use later on)
Krasnosel'skii in the 1951 \cite{Kras51} (and later on studied by
Borisovich in 1955; see further details in \cite[p.~358]{Kras})
avoided the transition to the projective spaces obtained by
identifying points of the sphere which are symmetric with respect
to the centre, needed to estimate the category of
Lusternik--Schnirel'man. In those terms, the genus of a set
provides us with a lower bound of the category. Moreover, it is
clear that an estimate of the number of critical points of a
functional is at the same time an estimate of the number of
eigenvectors of the gradient functional (in Krasnosel'skii's
terms) and, hence, of the number of solutions of the associated
nonlinear equation.

In our particular case, this functional subset is the following:
\begin{equation}
 \label{R0}
    \tex{ \mc{R}_{0,\g}=\Big\{v\in W^{2,1}_0(\O)\,:\, \int\limits_\O |\nabla v|^2 -
    \g \int\limits_\O |(-\D)^{-1/2}v|^2 =1\Big\}.}
\end{equation}
 Here, ``1" on the right-hand side plays no role and any positive constant would do.



According to the L--S approach (see \cite{Berger, KrasZ, PohFM},
etc.), in order to obtain the critical points of a functional on
the corresponding functional subset, $\mc{R}_{0,\g}$, one needs to
estimate the category $\rho$ of that functional subset. Thus, the
category will provide us with the number of critical points that
belong to the subset $\mc{R}_{0,\g}$. This  depends on the value
of the parameter $\g$. Namely, similar to \cite{PohFM, GMPSob},
the $\rho(\mc{R}_{0,\g})$ is given by the number of eigenvalues
(with multiplicities) of the corresponding linear eigenvalue
problem satisfying:
 \be
 \label{rho1}
  \tex{
\rho(\mc{R}_{0,\g})= \sharp \{\nu_\b < 1\}, \quad \mbox{where} }
\ee
\be
\label{rho2}
 \tex{
 -\D \psi_\b - \g (-\D)^{-1} \psi_\b= \nu_\b \psi_\b \,\,\,\,\mbox{in}
 \,\,\,\, \O, \quad \psi_\b=0 \,\,\, \mbox{on} \,\,\, \p \O
 }
 \ee
 (recall that the Navier condition $\D \psi=0$ is then valid
 automatically). Thus, by studying the eigenvalue problem
 \ef{rho2}, a sharp estimate of the category \ef{rho1} gets not
 that straightforward and easy, so we will need some extra
 analysis via embeddings of the corresponding functional spaces
 involved. However, some preliminary important conclusions from
 \ef{rho2} are indeed, possible. For instance, for $\g>0$ sufficiently large (in particular, $\g >\frac{1}{M_1}$), having on the
 right-hand side of \ef{rho2} special operators of different
 signs,
  $$
  -\D > 0 \andA -\g (-\D)^{-1}<0 \quad (\g >0)
  $$
  (and  the second one is ``weaker" in the  sense of compact
  embeddings), we have:
   \be
   \label{rho3}
    \fbox{$
\rho(\mc{R}_{0,\g}) \to + \iy \asA \g \to + \iy.
 $}
 \ee
 Since $\rho(\mc{R}_{0,\g})$ measures, at least, a lower bound of the total
 number of (L--S) solutions, \ef{rho3} clearly proves that an
 arbitrarily large number of various solutions can be achieved by
 enlarging the parameter $\g \gg 1$.

 \ssk

Therefore, by \cite{PohFM, GMPSob, GMPSob2} and as mentioned
above, if we look for critical points of the functional $\mc{G}_\g
(v)$ \eqref{spfunc} on the set $\mc{R}_{0,\g}$ it will be
necessary to estimate the category $\rho$ of that set
$\mc{R}_{0,\g}$.  The critical values $c_\b$ and the corresponding
critical points $\{v_\b\}$ are:
\begin{equation}
\label{cat} \tex{c_\b := \inf\limits_{\mc{A}\in \mc{A}_\b}
\sup\limits_{v\in \mc{A}} \mc{G}_\g (v)} \quad  (\b=1,2,3,...),
\end{equation}
where $\mc{G}_\g (v)$ is the functional defined by \eqref{spfunc}
and $$
  \mc{A}_\b :=\{\mc{A}\,:\, \mc{A}\subset
\mc{R}_{0,\g},\,\hbox{compact subsets},\quad \mc{A}=-\mc{A}\quad
\hbox{and}\quad \rho(\mc{A}) \geq \b\},
  $$
   is the class of closed
sets in $\mc{R}_{0,\g}$ such that, each member of $\mc{A}_\b$ is of
genus (or category) at least $\b$ in $\mc{R}_{0,\g}$. The fact that
$\mc{A}=-\mc{A}$ comes from the definition of genus
(Krasnosel'skii \cite[p.~358]{Kras}) such that, if we denote by
$\mc{A}^*$ the set disposed symmetrically to the set $\mc{A}$, $$
  \mc{A}^*=\{ v\,:\, v^*=-v\in \mc{A}\},
  $$
   then, $\rho(\mc{A})=1$
when each simply connected component of the set $\mc{A} \cup
\mc{A}^*$ contains neither of the pair of symmetric points $v$ and
$-v$. Furthermore, $\rho(\mc{A})=\b$ if each subset of $\mc{A}$ can
be covered by, a minimum, $\b$ sets of genus one, and without the
possibility of being covered by $\b-1$ sets of genus one.

In particular, we know that the genus of an $N$-dimensional sphere
is equal to $N+1$. Moreover, it is known that applying an odd
continuous transformation ${\bf B}$, that we define as admissible,
we find that $${\bf B}(-v)=-{\bf B} v,\quad \rho({\bf
B}\mc{A})\geq \rho(\mc{A}).$$ Hence, assuming the class of compact
sets $\mc{A}_\b$ as subsets of the form ${\bf B} S^{\b-1} \subset
\mc{R}_{\g,0}$, with $S^{\b-1}$ representing a suitable
sufficiently smooth $(\b-1)$-dimensional manifold (for example,
the sphere) then, we can assure that in the class $\mc{A}_\b$ can
occur sets of genus not less than $\b$, $$\rho(\mc{A}_\b)\geq \b=
\rho(S^{\b-1}),$$ because $\mc{A}_\b \subset {\bf B} S^{\b-1}$.
One cannot forget that there can also be other sets, on a
different class, of genus $\b$. As a consequence, and by
definition, we find that $$ c_1\leq c_2 \leq \cdots\leq
c_{l_{0,\g}},$$ with $l_{0,\g}=l_{0,\g}(\mc{R}_{0,\g})$ standing
for the category of $\mc{R}_{\g,0}$. Indeed, taking $\e>0$, by
definition of the critical values $c_{\b+1}$, we have that a set
$A_1 \in \mc{A}_{\b+1}$ exists, such that
  $$
  \tex{\sup_{v\in A_1} \mc{G}_\g
(v) < c_{\b+1} + \e.}
$$
 Hence, if $A_1$ contains a subset $A_0 \in
\mc{A}_\b$ such that
  $$
  \tex{\sup_{v\in A_0} \mc{G}_\g (v) \leq
\sup_{v\in A_1} \mc{G}_\g (v)< c_{\b+1} + \e,} \quad \mbox{and}
  $$
 $$
 \tex{c_\b
=\inf_{\mc{A}\in \mc{A}_\b} \sup_{v\in \mc{A}} \mc{G}_\g (v) \leq
\sup_{v\in \mc{A}_1} \mc{G}_\g (v)< c_{\b+1} + \e,}
   $$
then,
$$c_\b < c_{\b+1}.$$
   Roughly speaking, since the dimension of the sets
$\mc{A}$ belonging to the classes of sets $\mc{A}_\b$ increases
with $\b$, this guarantees that the critical points delivering
critical values \eqref{cat} are all different.


By the analysis carried out above via the fibering method to
obtain an algebraic estimate of the number of critical points for
the functional \eqref{funfb}, it follows that the category
$l_{0,\g}=l_{0,\g}(\mc{R}_{0,\g})$ of the set $\mc{R}_{\g,0}$ is
equal to the number of eigenvalues $\nu_\b$ of the linear
operator corresponding to the linear eigenvalue
problem \eqref{rho2} depending on the relation of the eigenvalues
$\nu_\b$ with respect to the parameter $\g$.

Moreover, if the quadratic form
$$\tex{ \int\limits_\O |\nabla v|^2 -
    \g \int\limits_\O |(-\D)^{-1/2}v|^2, \quad \hbox{with} \quad v\in W^{2,1}_0}
     $$
has Morse index $q>0$ and, since, we know that the functional
\eqref{funfb} is lower semicontinuous and coercive (proved below,
in the existence section), then the equation \eqref{fbnew} will
have at least $q$ distinct solutions \cite[Theorem~6.7.9]{Berger}.
Note that, the Morse index will be precisely the dimension of the
space where the corresponding form is negatively definite. This
includes all the multiplicities of the eigenfunctions involved in
the corresponding subspace providing a different approach for the
multiplicity of solutions.

Furthermore, recall that the L--S variational aspects of
construction of critical points can be closely  related to
structure of the ``essential zeros and extrema" of the basic
patterns $\{u_l\}$ of the equation \eqref{fb1}. Indeed, in the
standard elliptic second-order case (i.e., with no nonlocal term),
in  the simplest {\em radial} $N$-dimensional case, it is well
known that, by Sturm's Theorem, each solution $u_l(r)$, with
$r=|x| \ge 0$ corresponding to the genus $l \ge 1$ has precisely
$l-1$ zeros (sign changes)
or $l$  isolated local extrema points.
Here, we will
try to derive some related results for fourth-order elliptic
equations, though it cannot be done in such an impressive and
rigorous manner as for the second-order equations enjoying strong
Maximum Principle features.


Now, before computing the number of solutions in terms of  the
genus of the set $\mc{R}_{0,\g}$, let us calculate the
corresponding critical value $c_\b$ of the functional
\eqref{spfunc} on  $\mc{R}_{0,\g}$. Thus, we have that, for a
given solution of \eqref{fbnew} (critical point of \eqref{funfb})
$$
  v= C u \in \mc{R}_{0,\g},
   $$
   the following holds:
    $$
    \tex{ C=
\frac{1}{\sqrt{\int\limits_\O |\nabla u|^2 -
    \g \int\limits_\O |(-\D)^{-1/2}u|^2}}.}
    $$
 Hence, the corresponding critical value is as follows:
\be
\label{critv}
\tex{c_u = \mc{G}_\g (v)= \big(\frac{1}{2}-\frac{1}{p-1}\big) \frac{\big(\int\limits_\O |\nabla u|^2 -
    \g \int\limits_\O |(-\D)^{-1/2}u|^2\big)^{p-1}}{\big(\int\limits_\O |v|^{p+1}\big)^{\frac{2}{p-1}}}.}
\ee

Therefore, we arrive at the following possibilities:

\vspace{0.2cm}

\noindent {\bf Genus one.} For the parameter $\g \leq \frac{1}{M_1}$, we have
previously proved that there exists at least one solution. Hence, the
{\em genus} will be one. Indeed, taking a critical point denoted
by $u_1$ of the functional \eqref{funfb}, under the variational
assumptions established for the fibering method, we have that
\be
\label{genone}
u_1=r(v_1)v_1,
\ee


\noi such that stands for the critical point of the functional
\eqref{spfunc}. In other words, $v_1$ is the function in which the
subsequent infimum is achieved,
\be
\label{infimum} \tex{\inf \mc{G}_\g (v) \equiv \inf
\big(\frac{1}{2}-\frac{1}{p-1}\big) \frac{1}{\big(\int\limits_\O
|v|^{p+1}\big)^{\frac{2}{p-1}}}, \quad \hbox{with}\quad v_1 \in
\mc{R}_{0,\g}.} \ee Indeed, assuming that the domain
 $\O$ is large enough,
 let us take a two hump
structure (as done in \cite{GMPSob}) $$
 \hat{v}(x)=C[v_1(x)
+v_1(x+a)],\quad C\in \re, $$ with sufficiently large $|a|$,
If necessary, we also perform a slight modification of $\hat v(x)$
near the boundary to satisfy the boundary conditions.

 Thus, since $\hat{v}$ belongs to $\mc{R}_{0,\g}$ and,
hence, $C=\frac{1}{\sqrt{2}}$ (more precisely, $C \approx \frac
1{\sqrt 2}$), we have
   $$ c_{\hat{v}}=\mc{G}_\g (\hat{v}) = 2
\mc{G}_\g (v_1)> \mc{G}_\g (v_1)=c_{v_1}\equiv c_1.$$ Observe
that, since $\g\leq \frac{1}{M_1}$, we find that
 $$
 \tex{
 \big(\int\limits_\O |\nabla
u|^2 -
    \g \int\limits_\O |(-\D)^{-1/2}u|^2\big) >0.
 }
    $$
 Thus, for any $\hat{v}=M v_1$ with $v_1\in \mc{R}_{0,\g}$, such
that we  have that
$$
 \tex{
\big(\int\limits_\O |\nabla \hat{v}|^2 -
    \g \int\limits_\O |(-\D)^{-1/2}\hat{v}|^2\big) =M^2>1,
 }
    $$
and, hence,
$$
 \mc{G}_\g (\hat{v}) > \mc{G}_\g (v_1),
$$
 meaning that, in the present case, a two-hump structure cannot be a
 L--S one.


\vspace{0.2cm}

\noindent {\bf Genus greater than two.} In this particular case, we have proved
that the functional \eqref{funfb} has at least two solutions once
it is assumed that $\g> \frac{1}{M_1}$. Indeed, by that
condition over the parameter $\g$ one can say that the functional
\eqref{spfunc} is between at least two values, a maximum and a minimum one,
$$c_2 \leq  \mc{G}_\g (v)\leq c_2^*.$$
Therefore, we will obtain
at least two positive critical points for such a functional since the L--S characterization provides us with a lower bound for solutions but not exactly how many
are obtained, which confirms
our previous algebraic results. It should be pointed out that the situation in which there are infinitely many critical points is not ruled out.


To summarize, we state the following result:





\begin{lemma}
\label{lemsum}
The following possibilities for the number of critical points for
the functional $\mc{F}_\g(u)$ \eqref{funfb} hold:
\begin{enumerate}
\item[(i)] The elliptic problem \eqref{fb1} and, hence, \eqref{fbnew},
 admits an arbitrarily large number
   of different solutions  $u \in W_0^{2,2}(\O)$,
 provided that $\g> \frac{1}{M_1}$ such that
$M_1=\frac{1}{\l_1} \frac{p+1}{p}>0$, for sufficiently large
$\g\gg 1$.
\item[(ii)] Moreover, if $\g> \frac{1}{M_1}$ and we consider only positive solutions, i.e., positive critical points
of the functional \eqref{funfb}, then, there will be no solution.
\item[(iii)] And, finally,  if $\g \leq \frac{1}{M_1}$, there exists only one critical point
$r_1(v_1)v_1=u_1$ for the functional $\mc{F}_\g(u)$ that will be a local maximum.
\end{enumerate}
Each solution is obtained
as a critical point of the functional \eqref{funfb} in
$W_0^{2,2}(\O)$.
\end{lemma}



\ssk

\noi {\bf Remark on $\g$-bifurcation branches: reviving the total
number of solutions.} It is worth mentioning that those values of
the parameter $\g$ can be used to represent a family of nontrivial
solutions bifurcating from the branch of trivial solutions
$(0,\g)$. This fact can provide us with interesting information
about bifurcation near an eigenvalue of higher multiplicity. In
fact, we can expect that, for multiplicity $M$, we still have at
least $M$ distinct one-parameter families emanating from $(0,\g)$.
 Indeed, one can see from \ef{fb1} that those bifurcation values
$\g_\b$, numerated by a multiindex $\b$ in $\ren$,
 coincide with the eigenvalues of the bi-harmonic operator in
 \ef{eigp} (this time, with the original Navier boundary conditions).
Then classic bifurcation-variational theory \cite{Berger, KrasZ}
suggests that  each such $\g_\b$ is indeed a bifurcation point
from zero, and each such $\g$-branch (or a finite number of
branches in the multiple cases) can be extended to $\g > \g_\b$.
This again gives us a precise estimate of a number of various
solutions for large values of $\g$. Of course, this revives the
same conclusions obtained earlier by the L--S analysis and the
fibering method.

Indeed, using  Morse index theory, available since the functional
is coercive and lower semicontinuous, we know that, every time the
Morse index changes, there is a bifurcation point. Since the Morse
index is precisely the dimension of the space where the
corresponding functional is negatively definite, including all the
multiplicities of the eigenfunctions, we will obtain a bifurcation
point for every eigenvalue of the bi-harmonic operator
\eqref{eigp}. Also, if the Morse index is infinite there will be
an infinite number of bifurcation points, in clear concordance
with  \eqref{rho3}. This analysis can provide us with a different
approach in obtaining the multiplicity of the solutions for the
functional \eqref{funfb}.

Later on, we will discuss
some bifurcation ideas in a more complicated problem, associated
with non-potential operators.

\ssk




\subsection{Existence of solutions}

\noindent Firstly, we prove that the functional \eqref{funfb} is
coercive,  that is crucial in obtaining the existence of solutions
for the functional \eqref{funfb}. Thanks to the weakly lower
semicontinuity of the first two terms of the functional
\eqref{funfb}, we have that
\begin{equation*}
     \tex{\frac{1}{2} \int\limits_\O |\nabla u|^2 - \frac{\g}{2} \int\limits_\O |(-\D)^{-1/2}u|^2
     \geq K \left\|u\right\|_{W_0^{2,1}(\O)}}
\end{equation*}
for any $u \in W_0^{2,1}(\O)$. Note that, if $u \in \mc{C}_\g$,
then
\begin{equation}
\label{ext}
    \tex{\mc{F}_\g(u)=(\frac{1}{2}-\frac{1}{p+1}) \big(\int\limits_\O |\Delta u|^2 - \g \int\limits_\O |(-\D)^{-1/2}u|^2\big).}
\end{equation}
Therefore, by the weak lower semicontinuity via Lemma\;\ref{lem1},
we conclude that the functional \eqref{funfb} is coercive and
bounded below, with $1<p<\frac{N+2}{N-2}$ if
$N>2$ and any $p>1$ if $N=1,2$.

Consequently, due to the multiplicity results described in
Lemma\;\ref{lemsum}, the following theorem summarizes the
existence of non-zero solutions for the functional \eqref{funfb}:



\begin{theorem}
\label{exith}
 For any $p>1$,  the existence of solutions
for the boundary value problem \eqref{fb1}  is as follows:
\begin{itemize}
\item If $\g \leq \frac{1}{M_1}$, with $M_1=\frac{1}{\l_1} \frac{p+1}{p}>0$, then there exists
at least one solution;


\item If $\g >\frac{1}{M_1}$, then the elliptic problem \eqref{fb1} and, hence, \eqref{fbnew},
 admits an arbitrarily large number
   of different solutions  $u \in W_0^{2,2}(\O)$ provided that $\g \gg 1$.
   In  particular,
  there exists no positive
solutions of  \eqref{fb1}.
   \end{itemize}

\end{theorem}


\begin{remark}
{\rm All the solutions obtained for the boundary value problem
\eqref{fb1} are classical solutions by elliptic regularity for
higher-order equations (see Schauder's theory in \cite{Berger} for
further details).}
\end{remark}

\ssk

\noi{\em Proof.} Owing to the coercivity  of the functional
\eqref{funfb} and because it is also bounded below, and to the
weak lower semicontinuity, there exists a maximizing sequence
$\{u_n\}$ in $W_0^{2,1}(\O)$ for the functional \eqref{funfb} such
that
 $$
 \tex{\lim_{n \rightarrow \infty} \mc{F}_\g(u_n) = \sup_{u\in
\mc{C}_\g} \mc{F}_\g(u)>0,}
  $$
   and, hence, $\{u_n\}$ is bounded in
$W_0^{2,1}(\O)$. Then, by standard arguments, we can extract a
convergent subsequence, denoted again by $\{u_n\}$, such that $u_n
\rightharpoonup u_1$ as $n \rightarrow \infty$ for $u_1 \in
W_0^{2,1}(\O)$. In fact, such a convergence is strong in
$W_0^{2,1}(\O)$.

To this end, we argue by contradiction using the fibering maps and
the discussion made previously about the number of critical points
for such functions. Hence,  we arrive at the following situations,
either there will be no positive solution if $\g>\frac{1}{M_1}$
(the fibering map is strictly decreasing so there will not be any
critical point for the functional \eqref{funfb}), or there only
exists a classical solution if $\g \leq \frac{1}{M_1}$, with
$M_1=\frac{1}{\l_1} \frac{p+1}{p}>0$, or there are an arbitrarily
large number
   of different solutions  $u \in W_0^{2,2}(\O)$ if $\g > \frac{1}{M_1}$.

Namely, suppose that the strong convergence does not take place.
Therefore, since $u_n \in \mc{C}_\g$ and by the structure of
the functional \eqref{funfb}, we have that
\begin{equation}
\label{seqfun}
    \tex{ \mc{F}_\g(u_n):=(\frac{1}{2}-\frac{1}{p+1}) \int\limits_\O |u_n|^{p+1}.}
\end{equation}
By Lemma\;\ref{lemsum}, we know that there exists a maximum if $\g
< \frac{1}{M_1}$, where
 $$
 \tex{
 \int_\O |u|^{p+1}>0\quad \hbox{for }\quad  u\in W_0^{2,1}(\O).
 }
 $$
Hence,
passing in \eqref{seqfun} to the limit as $n \rightarrow \infty$
and using the weak semicontinuity of the third term of the
functional \eqref{funfb}, we find that, actually, $\int_\O
|u_1|^{p+1}>0$ and, consequently,
\begin{equation*}
    \tex{\mc{F}_\g(u_1)=\lim_{n \rightarrow \infty} \mc{F}_\g(u_n)= \sup_{u\in \mc{C}_\g} \mc{F}_\g(u),}
\end{equation*}
contradicting the nonexistence of a strong convergence in
$W_0^{2,1}(\O)$.

On the other hand, now, again by the coercivity of the functional
\eqref{funfb} and the weak lower semicontinuity, there exists a
minimizing sequence $\{u_n\}$ in $W_0^{2,1}(\O)$ for the
functional \eqref{funfb} so that $$\tex{\lim_{n \rightarrow
\infty} \mc{F}_\g(u_n) =\inf_{u\in \mc{C}_\g} \mc{F}_\g(u).}$$
Hence, by standard arguments, we can extract a convergent
subsequence, denoted again by $\{u_n\}$, so that  $u_n
\rightharpoonup u_1$ in $W_0^{2,1}(\O)$, to a certain $u_1 \in
W_0^{2,1}(\O)$.

Therefore, when $\g < \frac{1}{M_1}$, we have that there exists
$r_1$ such that $\mc{F}_\g(r_1(v)v)<0$. Hence,
\begin{equation}
\label{negfun}
\tex{\inf_{u\in\mc{C}_\g} \mc{F}_\g(u) <0,}
\end{equation}
since the fibering map is decreasing in an interval around the
value $r_1$,  which  the minimum is achieved at. Indeed, we can
write the functional $\mc{F}_\g(u)$ in \eqref{funfb} for any $u_n
\in \mc{C}_\g$
\begin{equation*}
    \tex{\mc{F}_\g(u_n)=(\frac{1}{2}-\frac{1}{p+1}) \Big(\int\limits_\O |\nabla u_n|^2
    - \g \int\limits_\O |(-\D)^{-1/2}u_n|^2\Big),}
\end{equation*}
and then
\begin{equation*}
    \tex{(\frac{1}{2}-\frac{1}{p+1}) \g \int\limits_\O |(-\D)^{-1/2} u_n|^2=
    (\frac{1}{2}-\frac{1}{p+1}) \int\limits_\O |\nabla u_n|^2 -\mc{F}_\g(u_n).}
\end{equation*}
So, passing to the limit as $n\rightarrow \infty$, we arrive at
$\g \int_\O |(-\D)^{-1/2}u_1|^2>0$. Moreover, since we are
assuming that the convergence of  $\{u_n\}$ is not strong in
$W_0^{2,1}(\O)$, we have that
\begin{equation}
\label{weaklim}
\tex{  \int\limits_\O |\nabla u_1|^2
    < \liminf_{n\rightarrow \infty} \int\limits_\O |\nabla u_n|^2.}
        \end{equation}
Note that the sign ``$\leq"$ is already obtained by the lower
semicontinuity. Thus, using the fibering maps \eqref{funsplit} we
have
\begin{equation*}
    \tex{ \phi_{v_n}'(r)=r\int\limits_\O |\nabla v_n|^2 -
    r\g \int\limits_\O |(-\D)^{-1/2} v_n|^2 - r^{p}\int\limits_\O
    |v_n|^{p+1},}
\end{equation*}
so that $u_n=r(v_n)v_n$ and
\begin{equation*}
    \tex{ \phi_{v_1}'(r)=r\int\limits_\O |\nabla v_1|^2 -
    r\g \int\limits_\O |(-\D)^{-1/2} v_1|^2 - r^{p}\int\limits_\O
    |v_1|^{p+1},}
\end{equation*}
with $u_1=r_1v_1$. From these expressions, it is easy to see that
$\phi_{v_n}'(1)=0$ for any $n$ ($u_n$'s are critical points of the
functional \eqref{funfb}) and thanks to the expression of
$\phi_{v_n}''(r)$ in \eqref{convex}, we also find that
$\phi_{v_n}'(r) <0$ for $0<r<1$. Consequently, applying
\eqref{weaklim} and the consequences explained previously yields
 $$
 \mc{F}_\g(r_1 u_1)< \mc{F}_\g(u_1)< \lim_{n \rightarrow \infty}
\mc{F}_\g(u_n) =\inf_{u\in \mc{C}_\g} \mc{F}_\g(u),
  $$
   which
contradicts the nonexistence of a strong convergence in
$W_0^{2,1}(\O)$.

Furthermore, using a different argument (but related to), owing to
Lemma\;\ref{lem3}, we know that the third term of the functional
$\mc{F}_\g(u)$ \eqref{funfb} is actually weakly semicontinuous if
$p<\frac{N+2}{N-2}$, i.e., taking the same convergent subsequence
in $W_0^{2,1}(\O)$,
\begin{equation*}
    \tex{\int\limits_\O |u_n|^{p+1} \rightarrow  \int\limits_\O
    |u_1|^{p+1}=N
    \quad \hbox{as} \quad n\rightarrow \infty.}
\end{equation*}
Thus, we only need to prove that the first two terms are actually
convergent. Namely,
\begin{equation*}
    \tex{  \int\limits_\O |\nabla u_1|^2 - \g \int\limits_\O |(-\D)^{-1/2}u_1|^2
    = \liminf_{n\rightarrow \infty} \big(\int\limits_\O |\nabla u_n|^2 - \g \int\limits_\O |(-\D)^{-1/2}u_n|^2\big).}
\end{equation*}
To do so, we argue again by contradiction, supposing that
\begin{equation*}
    \tex{  \big(\int\limits_\O |\nabla u_1|^2 - \g \int\limits_\O |(-\D)^{-1/2}u_1|^2\big)
    < \liminf_{n\rightarrow \infty} \big(\int\limits_\O |\nabla u_n|^2 - \g \int\limits_\O |(-\D)^{-1/2}u_n|^2\big),}
\end{equation*}
since, by the weak lower semicontinuity, we already know that the
sign  $``\leq"$ is achieved. Moreover, if $u_1$ is actually a
critical point of the functional \eqref{funfb},  by
\eqref{negfun}, we have
\begin{equation*}
    \tex{ \big(\int\limits_\O |\nabla u_1|^2 - \g \int\limits_\O |(-\D)^{-1/2}u_1|^2\big)
    \leq \int\limits_\O |u_1|^{p+1}=\hat N.}
\end{equation*}
Thus, assuming that the inequality is not true yields
\begin{equation*}
    \tex{ \big(\int\limits_\O |\nabla u_1|^2 - \g \int\limits_\O |(-\D)^{-1/2}u_1|^2\big)
    < \hat N.}
\end{equation*}
Then, it is possible to find $s_1>1$ such that $u_s=s_1 u_1$ and
\begin{equation*}
    \tex{ \Big(\int\limits_\O |\nabla u_s|^2 - \g \int\limits_\O |(-\D)^{-1/2}u_s|^2\Big)
    = \hat N.}
\end{equation*}
However,
\begin{equation*}
    \tex{ \int\limits_\O |u_s|^{p+1} =  s_1^{p+1} \int\limits_\O |u_1|^{p+1}= s_1^{p+1} \hat N > \hat N,}
\end{equation*}
which contradicts the assumptions. Hence, $u_n \rightarrow u_1$ in
$W_0^{2,1}(\O)$ and
\begin{equation*}
    \tex{\mc{F}_\g(u_1)=\lim_{n \rightarrow \infty} \mc{F}_\g(u_n)= \sup_{u\in \mc{C}_\g} \mc{F}_\g(u),}
\end{equation*}
which again contradicts the nonexistence of a strong convergence in
$W_0^{2,1}(\O)$.

To conclude the proof, one can combine these existence results
with those in Lemma\;\ref{lemsum} to arrive at the desired
assumptions of the theorem. \quad $\qed$




\ssk

\noi{\bf Remark.} Note that the first situation of existence of at
least one solution for the functional is consistent with the
conditions explained by \eqref{rex} for the existence of a one
turning point of the fibering map $\phi_v$, denoted by
\eqref{funsplit}. In other words,
\begin{align*}
    & \tex{\hbox{either}\quad \int\limits_\O |\nabla v|^2 - \g \int\limits_\O
    |(-\D)^{-1/2}v|^2>0\quad \hbox{and}\quad \int\limits_\O
    |v|^{p+1}<0 \,\,\,(\mbox{unavailable});} \\ & \tex{\hbox{or}\quad \int\limits_\O |\nabla v|^2 -
    \g \int\limits_\O |(-\D)^{-1/2}v|^2<0\quad \hbox{and} \quad\int\limits_\O |v|^{p+1}>0.}
\end{align*}



\subsection{The variational problem in $\ren$ in two cases: $\g>0$ and $\g <0$}

 \label{S.RN}

 In general, the results presented above can be
accomplished assuming posing the equation \eqref{fb1} in the whole
$\ren$ instead of in a bounded domain $\O \subset \ren$. However,
to do so, we need to consider the integrals over $\ren$ and the
functional setting over a certain weighted Sobolev space instead
of $W_0^{2,2}(\O)$ previously assumed. Such a functional setting
of the problem in $\ren$ is absolutely key in what follows.
Indeed, a proper functional setting assumes certain admissible
asymptotic decay of solutions at infinity, which, for \ef{fb1}, is
governed by the corresponding linearized operator.

\ssk

\noi\underline{\sc The case $\g >0$.}
 Assuming that $\g>0$ and, for
simplicity, the radial geometry, with $u=u(r)$, with $r = |x| \ge
0$, we then obtain, as $r \to \infty$,
 \be
 \label{inf1}
  \begin{matrix}
 \D^2 u \equiv u^{(4)}+ \frac {2(N-1)} r \,\,\, u'''+...=\g u+...
 \LongA  \ssk\ssk\\
  u(r)=C_1 \, {\mathrm e}^{- \g^{1/4} r} + r^{-
 \frac{N-1}2}\big[C_2 \cos (\g^{1/4}r)+ C_3 \sin
 (\g^{1/4}r)\big]+...
 \, ,
 \end{matrix}
 \ee
 where $C_{1,2,3}$ are arbitrary constants. Overall, we observe a
 3D bundle of solutions decaying at infinity, which looks rather
positive. However, one can see that the second term does not look
that good, and, in particular, for $N=1$, this represents
non-decaying oscillations as $r \to +\iy$, which do not belong to
any suitable functional space. For $N \ge 2$, these are decaying
but never belong to, say, $L^2(\ren)$ (!).


  Moreover, the {\em exponential} bundle obtained from \ef{inf1}
  for $C_2=C_3=0$ is {\em one-dimensional} only:
 \be
 \label{inf2}
  u(r)=C_1 \, {\mathrm e}^{- \g^{1/4} r}+... \asA r \to \iy, \quad C_1 \in \re.
  \ee
   Obviously, this 1D bundle is not
  enough to, say, ``shoot from infinity" {\em two} symmetry
  boundary conditions at the origin:
   \be
   \label{bc1}
   u'(0)=u'''(0)=0,
    \ee
    since, algebraically, at least two parameters are needed to
    satisfy \ef{bc1}.
Of course, the exponentially decaying solutions \ef{inf2} are the
best possible and belong to any reasonable functional space
naturally involved.


If the whole 3D bundle in \ef{inf1} is involved (note that this
can contradict any reasonable variational setting, but we are not
precise in that here), then the shooting problem becomes {\em
overdetermined}: using {\em three} parameters $C_{1,2,3}$ to
satisfy {\em two} boundary conditions \ef{bc1}. The set of
solutions is then expected not to be discrete and should be represented
via continuous curves. This case is more artificial and is less
interesting.

\ssk

\noi \underline{\sc The case $\g<0$.} This case is more promising.
Indeed, calculating the admissible asymptotics from \ef{inf1}
yields a {\em two-dimensional} exponential bundle:
 \be
 \label{inf1N}
  \tex{
 u(r)= {\mathrm e}^{-r |\g|^{1/4}/ \sqrt 2} \big[
C_1 \cos\big( \frac {|\g|^{1/4} }{\sqrt 2} r\big)+ C_2 \sin\big(
\frac {|\g|^{1/4} }{\sqrt 2} r\big)\big]+...\, , \quad C_{1,2} \in
\re. }
 \ee
 Matching with two symmetry boundary conditions \ef{bc1} yields a
 well-posed and well-balanced algebraic ``2D--2 shooting problem".
 A similar (but easier and without non-local terms) fourth-order
 problem was studied in \cite[\S~6]{GMPSob2}. It was shown that
 (for the analogy of the present case $\g <0$)
 such a problem can admit a countable set of countable families of
 solutions, where only the first infinite family is the L--S one. We expect
 that several properties and results of that study can be
 translated to the nonlocal problem under consideration, but this will require some additional work
 to be done in a separate paper \cite{CHPabloRN}.








\section{Global existence for the stable Cahn--Hilliard equation}
 \label{S5}

 Next, returning to the fourth-order non-stationary parabolic models presented at the beginning of this paper,
  without loss of generality, we consider the Cauchy problem for
the stable equation \ef{sCH}, with bounded and, if necessary,
exponentially decaying at infinity initial function $u_0(x)$.
 We are going to apply a scaling method, which gets rid in the limit of any
 lower-order or other perturbations in the PDE's, leaving only
 the main principal operators and nonlinearities that might be
 responsible for a finite time blow-up singularity. Therefore,
 to reveal key aspects of the method, we
 can consider this maximally simplified model \ef{sCH}.
 Our
main goal is to prove the following:

 \begin{theorem}
  \label{Th.Ex}
  The Cauchy problem $(\ref{sCH})$ in the parameter range
  $(\ref{sub1})$ admits a unique global classical solution, and
  moreover, it is uniformly bounded:
   \be
   \label{pp1}
   |u(x,t)| \le C \inB \ren \times \re_+.
    \ee
 \end{theorem}

 \noi {\em Proof.} It consists of  four steps.

 \noi{\sc Step I: a priori bounds on smooth solutions}.
 This is a pretty standard step in nonlinear PDE theory. Writing
 \ef{sCH} in the pseudo-parabolic form,
  \be
  \label{ps1}
  (-\D)^{-1} u_t= \D u -  |u|^{p-1} u,
   \ee
   and multiplying by $u_t$ in the metric of $L^2(\ren)$ yields:
    \be
    \label{ps2}
     \tex{
    \frac{\mathrm d}{{\mathrm d}t}\,\Big(\,
      \frac 12 \|\n u\|^2_2 + \frac 1{p+1}\,  \int\limits
      |u|^{p+1}\,\Big)
= - \|u_t\|^2_{H^{-1}} \le 0. }
 \ee
 In particular, this shows that, on smooth solutions, \ef{sCH} is
 a gradient dynamical system admitting a positive definite
 Lyapunov function, so that a number of strong results from this
 area are available (see e.g., \cite{Hale}), though the key
 $L^\iy$-estimate still remains uncertain.

 Integrating \ef{ps2} over an arbitrary interval $(0,T)$ yields the
 following {\em a priori} bounds on smooth solutions:
  \be
  \label{ps3}
   \tex{
   \| \n u(t)\|^2_2 \le C \andA \int
      |u(t)|^{p+1} \le C \quad \mbox{for all} \quad t>0.
      }
       \ee

\ssk

\noi{\sc Step II: proving non-blow-up by scaling. Subcritical
range.}
 Here, we follow \cite{GW2}; see also \cite{GMPKS}.
 Namely,   arguing by contradiction, we  assume that  there exist sequences $\{t_k\}
 \to T^-$, $\{x_k\} \subset \ren$, and $\{C_k\}$ such that
 \be
 \label{seq11}
 \sup_{\ren \times [0,t_k]} \, |u(x,t_k)| = |u(x_k,t_k)| = C_k \to + \infty.
 \ee
 In other words, the solution blows-up in finite time.
 We next perform the change
 \be
 \label{rvarl}
u_k(x,t) \equiv v(x_k +x, t_k+t) = C_k v_k(y ,s) \whereA x = a_k
y, \quad t = a_k^{4}s,
 \ee
and the sequence  $\{a_k\}\to 0$ is chosen in such a manner that
the {\em a priori} estimates \ef{ps3} hold for the sequence
$\{v_k(y,s)\}$. In particular,
$$
   \tex{
   \| \n u(t)\|^2_2 = C_k^2 a_k^{N-2} \int \| \n v(s)\|^2 \andA \int
      |u(t)|^{p+1} =C_k^{p+1} a_k^N \int
      |v(s)|^{p+1}.
      }
 $$
 This gives respectively (both choices eventually lead to the same result):
  \be
  \label{ps4}
  a_k= C_k^{-\frac 2{N-2}} \,\,(N \ge 3) \andA
  a_k=C_k^{-\frac{p+1}{N}}.
   \ee
 Note that, after such a scaling, the rescaled functions
 $v_k(y,s)$ are defined on the intervals
  \be
 \label{ww1}
  \tex{
 \quad s \in \tex{\big[-\frac{t_k}{a_k^{4}},
 \frac{T-t_k}{a_k^{4}}\big)}.
  }
  \ee

 As usual, such a rescaling near blow-up time, in the limit $k \to \iy$, leads to
the so-called  {\em ancient solutions} (i.e., defined for all
$s<0$) in Hamilton's notation \cite{Ham95}. Various scalings  have
been  typical techniques of
   reaction-diffusion  theory for many years; see different forms of its application in
 \cite{SGKM, AMGV}.

  Substituting (\ref{rvarl}) into equation (\ref{sCH}) yields that
$v_k(y,s)$ satisfies a perturbed equation
 \begin{equation}
 \label{reql}
  (v_k) _s =
  -\D^2 v_k + \d_k \D( |v_k|^{p-1}v_k) \quad \mbox{in} \quad \ren \times
  \re,\quad \mbox{where}
  \ee
 \be
 \label{ps6}
  \tex{
 \d_k= C_k^{\g_1}, \,\, \g_1=p-1-\frac{4}{N-2} \andA
  \d_k= C_k^{\g_2}, \,\, \g_2=p-1- \frac{2(p+1)}{N},
   }
   \ee
   respectively.
One can see that
 \be
 \label{ps7}
 \d_k \to 0, \,\,\, \mbox{if} \,\, \g_{1,2}<0 \LongA p<p_*.
  \ee

 We next perform a backward shifting in time technique by
 fixing $s_0>0$ large enough (this is possible in the time-interval in (\ref{ww1})
  since $a_k \to 0$),
  and setting $\bar v_k (s) = v_k
(s-s_0)$. Then, by construction, we have that
 \be
 \label{sc44}
|\bar v_k (s)| \le 1 \andA  \mbox{(\ref{ps3}) for $v_k(y,s)$ hold
on} \,\,\, (0,s_0),
 \ee
 so that $\{\bar v_k(s)\}$ is a family of
uniformly bounded classical solutions of the uniformly parabolic
equation (\ref{reql}) with bounded smooth coefficients. By classic
parabolic regularity theory \cite{EidSys, Fr},
 we have that
 the sequence $\{\bar v_k\}$ is uniformly bounded and  equicontinuous on
 any compact subset of $\ren \times (0,s_0)$. Indeed,
 the necessary uniform gradient bound can be obtained from the
 integral equation of (\ref{reql}), or by other usual regularity methods
 for uniformly parabolic equations.

 Therefore, by the Ascoli--Arzel\'a theorem,
along a certain subsequence, $\bar v_k (s) \to \bar v (s)$
uniformly on compact subsets of $\ren \times (0,s_0)$. Passing
to the limit in equation (\ref{reql}) and using that the scaling
parameter satisfy $\d_k \to 0$,  yields that $\bar v(s)$ is a
bounded weak solution and, hence, a classical solution of the
Cauchy problem for the linear {\em bi-harmonic equation}
 \be
  \label{ww2}
   \tex{
  \bar v_s= -\D^2 \bar v, \quad \mbox{with data} \quad
  |\bar v_0| \le 1, \quad
\|\n \bar v_0\|_2 \le C, \,\,\, \int  |\bar v_0|^{p+1} \le C.
 }
 \ee
  We next represent the solutions as follows:
  \be
  \label{ww3}
   \tex{
   \bar v(s_0)= b(s_0)* \bar v_0 \equiv
    s_0^{-\frac N4} \int\limits_{\ren} F\big(\frac{y-z}{s^{1/4}}\big) \,
    \bar v_0(z)\, {\mathrm d}z
    \equiv  s_0^{-\frac N4} \int\limits_{\ren} (\n)^{-1} F\big(\frac{y-z}{s^{1/4}}\big) \,
   \n \bar v_0\, {\mathrm d}z.
    }
     \ee
 Finally, using the H\"older inequality in the convolution  yields:
 \be
 \label{ps8}
  |\bar v(s_0)|
\le s_0^{-\frac {N-1}4} C \ll 1 \quad (N \ge 2)
 \ee
    for all $s_0 \gg 1$. Hence, the same holds for
$\sup_y|\bar v_k(y,s_0)|$ for $k \gg 1$, from whence comes the
contradiction with the assumption  $\sup_y|v_k(y,s_0)| =1$.
  Thus, $v(x,t)$ does not blow-up and remains bounded for all $t>0$
   (but not uniformly still, as required by
   \ef{pp1}).

\ssk

\noi{\sc Step III: proving non-blow-up by scaling. Critical case.}
For $p=p_*$, we have that $\d_k \equiv 1$, so that \ef{reql} for
$\{v_k\}$ takes the unperturbed form
 \be
 \label{ps10}
  v _s =
  -\D^2 v +  \D(|v|^{p-1}v) \quad \mbox{in} \quad \ren \times
  \re.
  \ee
 Since, as we have seen, \ef{ps10} is a smooth gradient system
 with a monotone operator in $H^{-1}$, so that zero is the only
 equilibrium, we have that,
for any regular enough global solution $v(y,s)$,
 \be
 \label{ps12}
  v(y,s) \to 0 \asA s \to + \iy
  \ee
  uniformly in $\ren$. Then,
  as above,
by passing to the limit $k \to \iy$, we then obtain existence of
an {\em ancient} solution $\bar v(y,s)$ satisfying
 \be
 \label{ps13}
 \bar v(s) \to 0 \asA s \to - \iy \andA \|\bar v(0)\|_\iy =1.
  \ee
  However, such a solution is obviously nonexistent, since
then  \ef{ps10} for $s \ll -1$, where $|\bar v(s)| \ll 1$, becomes
an asymptotically small perturbation of the linear equation
\ef{ww2}, so that the same argument applies.

\ssk

\noi{\sc Step IV: uniform boundedness.} Assuming now that $C_k \to
+\iy$ and $t_k \to +\iy$ and performing the same scaling and
passing to the limit yield the result.

This completes the proof of the theorem. $\qed$

\ssk

\noi{\bf Remark.} For the non-autonomous C--H equation \ef{sCHa},
one can derive similar global {\em a priori} estimates \ef{ps2}.
However, since the translations in $x$ are not allowed now, we can
perform a proof of non-blow-up at the fixed point $x=0$. Then, all
the arguments apply if we set $x_k=0$ (or $x_k \approx 0$ close
enough),  and the only difference is that $\a$ will enter some
exponents, so that we will have
 $$
 \tex{
  a_k= C_k^{- \frac{p+1}{\a+N}}, \quad
  \g_1= p-1-\frac{2(\a+2)}{N-2}, \quad
  \g_2=p-1- \frac{(p+1)(\a+2)}{\a+N}, \quad \mbox{etc.},
  }
  $$
  which will eventually lead to the critical exponent in
  \ef{sub1a}. Indeed, this non-blow-up is a conventional one,
  since \ef{sub1a} does not prevent blow-up at any neighbouring
  points $x \not = 0$, for which the range \ef{sub1} remains
  correct.

 \section{A short discussion on blow-up in the supercritical stable
 model}
  \label{S6}

 Thus, in the supercritical range $p>p_*$, the scaling argument
 establishing uniform $L^\iy$-bounds \ef{pp1} from weaker Sobolev and
 $L^{p+1}$-estimates \ef{ps3} does not apply, and finite time
 blow-up of some solutions becomes plausible.

Consider the non-autonomous model \ef{sCHa}. Then blow-up at $x=0$
is impossible in the range \ef{sub1a}.
 As a first step towards blow-up scenarios (at $x=0$) for \ef{sCHa} with $p > p_*(\a)$, one
 should consider self-similar solutions of the standard form:
  \be
  \label{dd1}
   \tex{
   u_{\rm S}(x,t)=(T-t)^{-\g} f(y), \quad \g= \frac{\a+2}{4(p-1)}, \quad y= x/(T-t)^{\frac 14},
    \quad t<T,
    }
    \ee
  where $T>0$ is the corresponding blow-up time. Then $f$ in \ef{dd1} solves
  the following elliptic equation:
   \be
   \label{dd2}
    \tex{
   - \D^2 f- \frac 14 \, y \cdot \n f- \g f+ \D(|y|^\a
   |f|^{p-1}f)=0 \inB \ren.
    }
    \ee

Thus,
we arrive at a ``stationary" problem \ef{dd2} for blow-up profiles
$f(y)$. This returns us to previous Section \ref{S3}. However, one
can immediately observe that the corresponding steady problem is
not variational, due to the presence of an extra linear
first-order operator, increasing the difficulty of the analysis
and not allowing the techniques used in Section \ref{S3}.

Moreover, proving nonexistence of
any nontrivial solutions of \ef{dd2} in some subrange of
$p>p_*(\a)$ is an important but still difficult open
problem, at least in general.

However, there exist other more involved scenarios of blow-up,
which we now start to develop for the corresponding unstable C--H
equation. For \ef{ch6}, such scenarios are ``more generic" and
easier to implement, though, partially, in a non-rigorous way. We
still do not know whether such scenarios can be applied to the
stable equation \ef{sCHa}.

\section{Two types of blow-up in the unstable C--H equation}
 \label{S7}

 For simplicity, we now again fix $\a=0$ and consider the corresponding
 unstable C--H equation \ef{ch6}.
 The first type of possible blow-up patterns remains the
 same and will be shown as follows.

 \subsection{Self-similar blow-up}

 This occurs according to formulae as in \ef{dd1} with $\a=0$:
  \be
  \label{72}
   \tex{
   u_{\rm S}(x,t)=(T-t)^{-\g} f(y), \quad \g= \frac{1}{2(p-1)}, \quad y= x/(T-t)^{\frac 14},
   \quad t<T,
    }
    \ee
  where the similarity profile $f(y)$  solves
  the  elliptic equation
   \be
   \label{73}
    \tex{
   - \D^2 f- \frac 14 \, y \cdot \n f- \frac{1}{2(p-1)}\, f- \D(
   |f|^{p-1}f)=0 \inB \ren.
    }
    \ee

 Let us discuss properties of blow-up similarity profiles in
 greater detail, which is necessary for future extensions.
  Thus:

    (i) In the critical ``mass Fujita" case \ef{p01}, equation \ef{73}
    admits a countable family of positive blow-up patterns $\{f_k(y)>0\}_{k \ge 0}$,
    with exponential decay as $y \to \infty$.
    In the radial setting, these  are obtained from a third-order ODE
    derived on integration \cite[\S~2-4]{EGW}. This is  a simpler case.
     Moreover,  such Type I blow-up patterns generate Dirac's
 delta as final-time profiles,
 \be
 \label{Di1}
 u(x,T^-) = c_k \d(x) \inB \ren \quad (c_k>0).
  \ee

\ssk

(ii) There exists another critical dipole exponent, e.g., $p_1=2$
for $N=1$ (or $p_1= 1+ \frac 2{N+1}$), where the ODE for $N=1$
again reduces to the third order and admits extended study
\cite[\S~5]{EGW}.

\ssk


(iii) For general $p>1$, $p \not = p_0, \, p_1$, \ef{73} is truly a
fourth-order elliptic equation or an ODE in the radial setting.
Then construction of proper blow-up profiles $f(y)$ requires
taking into account non-exponentially decaying asymptotic bundles.
E.g., for $N=1$, this means considering a 2D bundle of the form
  \be
   \label{f23}
 \mbox{$
 f(y) = Ay^{-\frac 2{p-1}}+...+ C y^{-\frac 13} {\mathrm e}^{-a_0
 y^{4/3}}+... \quad \mbox{as} \,\,\, y \to + \infty,
 \quad a_0=3 \cdot 2^{-\frac 83},
 $}
  \ee
 where $A$ and $C$ are arbitrary parameters.
 In Appendix A,  we discuss a 2D shooting strategy, which requires
 constructing a solution of \ef{73} for $p \not = p_0$ by using both parameters $A$ and $C$ in \ef{f23}.

In \ef{f23}, the last term is responsible for exponentially
decaying functions, while the first one gives an algebraic decay.
 It follows from \ef{f23} that, for any $A \neq 0$,
  \be
  \label{f67}
  f \in L^1(\re) \forA p \in (1,p_0) \andA f \not\in L^1(\re) \forA p
  >p_0.
   \ee
Since the mass evolution of the similarity solutions \ef{72} is
given by
 \be
 \label{f98}
  \tex{
   \int\limits_{\ren} u_{\rm S}(x,t)\, {\mathrm
   d}x=(T-t)^{\frac{N(p-p_0)}{4(p-1)}} \, \int\limits_{\ren}
   f(y)\,{\mathrm d}y,
   }
   \ee
   the mass conservation implies that
 \be
 \label{f99}
  \tex{
   \int f(y) =0 \quad \mbox{for all} \quad p \in(1,p_0).
    }
    \ee

Thus, the following holds:
 $$
  \tex{
 A=0 \,\,\, \mbox{in (\ref{f23})},\quad \mbox{iff} \quad p=p_0= 1 + \frac 2N,
 }
 $$
 i.e., in other cases, $f(y)$ cannot in general have exponential
 decay.
 This  changes the blow-up asymptotics for $p \not = p_0$:  \ef{72},
 \ef{f23} imply  the limit $t \to T^-$
  \be
  \label{f24}
  u(x,T^-)= A |x|^{- \frac 2{p-1}}, \quad A \ne 0, \quad \mbox{so
  that}
  \ee
   \be
   \label{f25}
   u(x,T^-) \in L^1_{\rm loc}(\ren), \quad \mbox{if} \quad p>p_0;
   \quad
   u(x,T^-) \not\in L^1_{\rm loc}(\ren), \quad \mbox{if} \quad
   p<p_0.
    \ee
 We refer to \cite{EGW}, where most of the results are obtained for
 $N=1$. In Figure \ref{FF1}, we present numerical results showing
 how the blow-up similarity profiles $f(y)$ changes sign for $p <p_0$,
 so that we expect that
  \be
  \label{AA1}
  A(p)<0 \forA p\in (1,p_0), \quad A(p_0)=0, \andA A(p)>0 \forA
  p>p_0.
  \ee

\begin{figure}[!htb]
\begin{center}
\includegraphics[width=10cm]{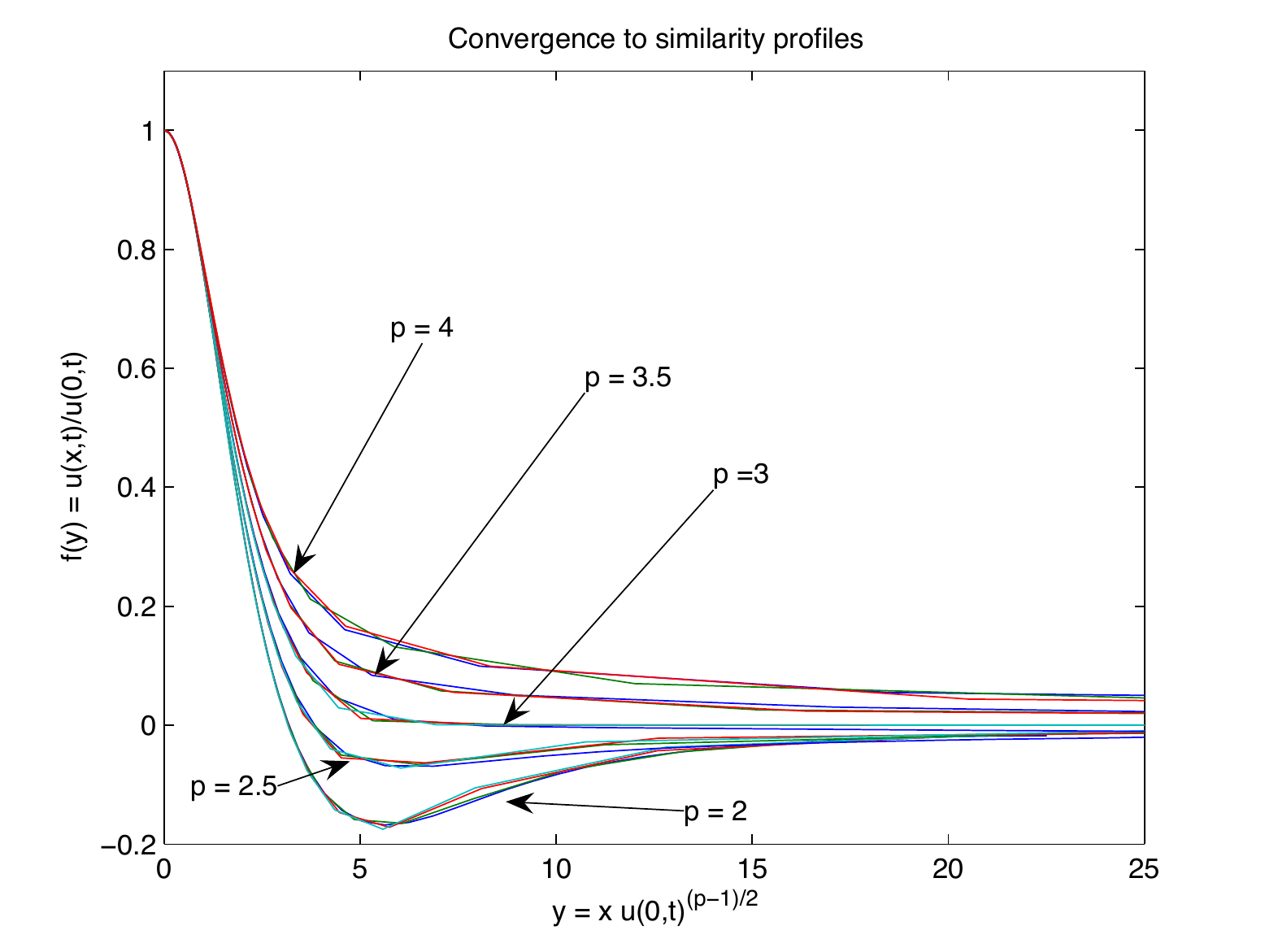}
\caption{Similarity profiles satisfying \ef{73} for $N=1$ obtained
via numerical solving  the PDE \ef{ch6} and scaling,
\cite[\S~5]{EGW}.}
 \label{FF1}
\end{center}
\end{figure}

  Main ideas and shooting techniques of construction admit  natural extensions to higher
 dimensions, where mathematical justifications become much more
 difficult, especially in the critical and supercritical Sobolev
 range
  \be
  \label{75}
   \tex{
    p \ge p_{\rm S}= \frac {N+2}{(N-2)_+}.
    }
     \ee
In particular, it can be expected that these $p$-branches of
similarity profiles may blow-up as $p \to p_{\rm S}^-$, without
proper extension beyond. Therefore, for $p \ge p_{\rm S}$, another
scenario of blow-up is necessary, to be introduced later on.

\subsection{A branching approach to blow-up similarity profiles}
 \label{S.6.2}

Indeed, the problem \ef{73} is not variational. However, it can be
viewed as a perturbation of a variational one, which we have dealt
with above. Let us introduce the following family of operators:
\be
   \label{73N}
    \tex{
    {\bf A}_\mu f \equiv
   - \D^2 f- \mu \, y \cdot \n f- \frac{1}{2(p-1)}\, f- \D(
   |f|^{p-1}f)=0 \inB \ren,
    }
    \ee
     where $\mu \in [0, \frac 14]$ is a parameter. Indeed, for
     $\mu=0$, we arrive at the variational  problem \ef{fb1} in $\ren$, with
     (this is key for existence!)
      \be
      \label{bbb1}
       \tex{
       \g=- \frac 1{2(p-1)}<0 \quad (\mbox{cf. Section \ref{S.RN}}).
       }
       \ee
Therefore, we expect that there exists a {\em branching} of
solutions of \ef{73N} from critical L--S points at $\mu=0$. This
can be established in a reasonably standard way (see related
examples in \cite{GMPSob, GMPSob2}). This can also then guarantee
existence of an arbitrarily large number of solutions (if the L--S
family for $\mu=0$ is countable) at least for sufficiently small
$\mu>0$. The principle difficulty is then a  extension of these
solutions up to the necessary value $\mu= \frac 14$ dictated by
the similarity blow-up equation \ef{73}. Such a global extension
problem remains open and represents a fundamental problem in
nonlinear non-potential operator theory. Numerical methods even
for the case $N=1$ (or the radial one in  $\ren$) are also rather
delicate in the present case. We will treat such a problem in
\cite{CHPabloRN}.



 \subsection{Remark on similarity extension beyond blow-up: towards  Leray's
scenario (1934)}
  Let us point out another important feature of
self-similar blow-up under the presence of the mass conservation.
 Namely, self-similar blow-up such as \ef{72} admits
 self-similar extensions beyond via {\em global} similarity solutions
 \be
  \label{76}
   \tex{
   u^+_{\rm S}(x,t)=(t-T)^{-\g} F(y), \quad \g= \frac{1}{2(p-1)}, \quad y= x/(t-T)^{\frac 14},
   \quad t>T,
    }
    \ee
  where the similarity profile $F$
  now solves a slightly different
     elliptic equation
   \be
   \label{77}
    \tex{
   - \D^2  F+ \frac 14 \, y \cdot \n F + \g F- \D(
   |F|^{p-1}F)=0 \inB \ren.
    }
    \ee
 Then, for ``initial data" \ef{Di1}, we look for solutions
 of \ef{77} with exponential decay, while, for that in \ef{f24}, one
 needs $F(y)$ with the {\em same} constant $A \ne 0$, as inherited
 from blow-up evolution via $f(y)$ with the asymptotics \ef{f23}.

It is key that the ODE \ef{77} in the radial setting has another
3D bundle at infinity, e.g., for $N=1$,
 \be
   \label{f23N}
 \mbox{$
 F(y) = Ay^{-\frac 2{p-1}}+...+  y^{-\frac 13} {\mathrm e}^{-\frac
 {a_0}2
 y^{4/3}}\big[B \cos\big( \frac{a_0 \sqrt 3}2 y^{\frac 43}\big)+
 C\sin\big( \frac{a_0 \sqrt 3}2 y^{\frac 43}\big) \big]+...\,,
  $}
   \ee
where $A \ne 0$ is fixed by the pre-history as $t \to T^-$, but
two parameters $B$ and $C$ are left to shoot necessary two
symmetry conditions at the origin
 $$
 F'(0)=F'''(0)=0.
 $$

Note that, in general, such a construction beyond blow-up can
assume existence of a certain ``mass defect", when the parameters
for $t<T$ and $t>T$ {\em are not} entirely consistent; see
\cite[\S~4]{EGW} and \cite{GalJMP} for the case $p=p_0$. Indeed,
solvability of the above shooting/matching problem on construction
of suitable {\em extension pairs} $\{f(y),F(y)\}$ for blow-up
solutions of the C--H equation must be accompanied by advanced
numerical methods. For the fourth-order RD equation in 1D (a
non-conservative model)
 \be
 \label{RD4}
 u_t= -u_{xxxx} + |u|^{p-1}u \inB \re \times \re,
  \ee
for such a blow-up/extension study, see \cite{BGW1, Gal5types,
GalBlExt}.

\ssk

     This construction can be connected with Leray's scenario of self-similar blow-up/extension
 proposed in 1934 for the Navier--Stokes equations in $\re^3$,
 \cite[p.~245]{Ler34}.

\ssk

     Thus, we next turn our attention to another new type of
     blow-up for the unstable C--H equation \ef{ch6}, which is directly associated with its {\em
     stationary} solutions.

  \subsection{``Quasi-stationary" {\bf Type II(LN)} blow-up in the critical case}

We consider the unstable equation \ef{ch6} in the {\em critical
Sobolev case}
 \be
 \label{76S}
  \tex{
  p=p_{\rm S}= \frac{N+2}{N-2} \forA N \ge 3.
  }
  \ee
Then, the stationary equation
 \be
 \label{6.1}
  \tex{
 \D_\xi W + W^p=0, \quad \xi \in \ren, \,\,\, W(0)=d>0 \quad \big( p=p_{\rm
 S}=\frac{N+2}{N-2}\big)
 }
  \ee
 is known to admit a 1D family of
classic Loewner--Nirenberg (L--N) conformally invariant exact
solutions \cite{LN74} (1974).  The corresponding symmetries of
\ef{6.1} were earlier detected by
  Ibragimov in 1968 \cite{Ib68}. These solutions are given explicitly by
 \be
 \label{6.2}
  \tex{
  W_0(\xi) = d\, \Big[ \frac{N(N-2)}{N(N-2)+ d^{4/(N-2)}|\xi|^2}
  \Big]^{\frac{N-2}2}>0 \inB \ren, \quad d>0,
  }
   \ee
 and exhibit a number of uniqueness and other exceptional
 properties concerning  equation \ef{6.1}; see \cite[\S~6]{Gal5types}
 for extra details and references.

The idea of such {\bf Type II(LN)} blow-up patterns (according to
a classification in \cite{Gal5types}) consists of noting that
blow-up can occur via some ``slow" motion about the stationary
solutions \ef{6.2}. This formally means the following
non-stationary parameter time-dependence:
 \be
 \label{79}
d=d(t) \to + \iy \asA t \to T^-.
 \ee
 For the standard second-order reaction-diffusion equation
  \be
  \label{80}
  u_t=\D u+ |u|^{p-1}u, \quad p=p_{\rm S},
   \ee
 such a blow-up scenario was  proposed in \cite{Fil00}.


Construction of such Type II blow-up patterns of non-self-similar
kind consists of few steps.

\ssk

\noi{\sc Step I: standard rescaling}.
 We first  use standard full similarity scaling in \ef{ch6}:
 \be
 \label{3.1}
  \tex{
 u(x,t)=(T-t)^{-\frac 1{2(p-1)}}v(y,\t), \quad y= \frac
 x{(T-t)^{1/4}}, \quad \t = - \ln(T-t) \to +\iy, \,\, t \to T^-.
}
  \ee
 Then $v(y,\t)$ solves the following rescaled parabolic equation:
  \be
  \label{3.2}
   \tex{
  v_\t= {\bf A}(v) \equiv - \D^2 v - \frac 14\, y \cdot \n v-\frac 1{2(p-1)}\, v
  - \D(
  |v|^{p-1}v)\inB \ren \times (\t_0,\iy),
  }
   \ee
   where $ \t_0=-\ln T$, and ${\bf A}$ is the stationary elliptic operator in
   \ef{73}, so that similarity profiles  are just
   stationary solutions of \ef{3.2}.

\ssk

\noi{\sc Step II: spectral theory and generalized Hermite
polynomials}. We need a detailed knowledge of spectral properties
of the non-self-adjoint linear operator, which appeared in
\ef{3.2}, i.e., for
  \be
  \label{4.2}
  \tex{
  \BB^*=-\D^2 - \frac 1{4}\, y \cdot \n \inB L^2_{\rho^*}(\ren),
  \,\,\,\rho^*(y)={\mathrm e}^{-a|y|^{4/3}}, \quad a \in \big(0,3 \cdot
 2^{- \frac {8}3}\big).
  }
  \ee
  This ``adjoint"
    Hermite operator has a number of good spectral
  properties \cite{Eg4}:

  \begin{lemma}
\label{lemSpec2} 
  $\BB^*: H^{4}_{\rho^*}(\ren) \to
L^2_{\rho^*}(\ren)$ is a bounded linear operator with the spectrum
 \be
\label{SpecN}
 \tex{
 \s(\BB^*) = \{ \l_\b =  -\frac{|\b|}4, \,\,\,
|\b|=0,1,2,...\} \quad \big(= \s(\BB), \,\, \BB=-\D^2 + \frac
1{4}\, y \cdot \n+ \frac N4 \, I \big). }
   \ee
     Eigenfunctions
$\psi_\b^*(y)$ are $|\b|$th-order generalized Hermite polynomials:
\be
\label{psidec}
 \tex{
\psi_\b^*(y) =  \frac {1}{ \sqrt{\b !}}\big[ y^\b +
\sum_{j=1}^{[|\b|/4]} \frac 1{j !}(\Delta)^{2 j} y^\b \big],
 \quad |\b|=0,1,2,...\, ,
 }
  \ee
and the subset $\{\psi_\beta^*\}$ is complete
 in
$L^2_{\rho^*}(\ren)$.

\end{lemma}

As usual, if $\{\psi_\b\}$ is the adjoint basis of eigenfunctions
of the adjoint (in the dual metric of $L^2$) operator
 \be
 \label{BBB}
 \tex{
 \BB= - \D^2 + \frac 14\, y \cdot \n + \frac N4\, I \inB
 L^2_\rho(\ren), \quad \rho= \frac 1{\rho^*},
 }
\ee
   with the same spectrum \ef{SpecN}, the bi-orthonormality
condition holds in $L^2(\ren)$:
 \be
 \label{Ortog}
\langle \psi_\mu, \psi_\nu^* \rangle = \d_{\mu\nu} \quad \mbox{for
any} \quad \mu, \,\, \nu.
 \ee
We recall that the generating formula for eigenfunctions $\psi_\b$
is as follows \cite{Eg4}:
 \be
 \label{eig1}
  \tex{
   \psi_\b(y)= \frac {(-1)^{|\b|}}{\sqrt{\b !}}\, D^\b F(y),
   }
   \ee
 where $F(y)$ is the rescaled kernel of the fundamental solution
 of the bi-harmonic equation
   \be
   \label{bi1}
  u_t= - \D^2 u \inB \ren \times \re_+, \quad \mbox{so that}
    \ee
  \be
  \label{eig2}
   \tex{
  b(x,t)= t^{-\frac N4} F(y), \quad y= \frac x{t^{1/4}},
  \quad \mbox{and} \quad \BB F=0, \quad \int F=1.
   }
   \ee

\noi{\sc Step III: formal construction of Type II(LN) blow-up
patterns for $p=p_{\rm S}$}.
Let $v(y,\t)$ be the rescaled solution of \ef{3.2} in, say, radial
geometry at the moment (which is not essential, since the patterns
can be essentially non-radial). This is the  precise specification
of the class of blow-up patterns we are dealing with: we assume
that $v(y,\t)$ behaves for $\t \gg 1$
 closely to the stationary manifold composed of the explicit  equilibria
\ef{6.2}, i.e., for some unknown function  $\var(\t)$, the
``singular part" of the solution takes the form
 \be
 \label{sc12}
  v(y,\t) = \var(\t) W_0\big(\var^{\frac{p-1}2}(\t)y\big)+...\, ,
   \quad \mbox{where} \quad \var(\t) \to +\iy \asA \t \to + \iy.
 \ee
 Such a representation \ef{sc12} of $v(y,\t)$ is then assumed to be uniformly
 valid as $\t \to +\iy$
on any compact subsets in the new variable
$\zeta=\var^{\frac{p-1}2}(\t)y$.

Next,  it  follows that, on the solutions \ef{sc12} in terms of
the original rescaled variable $y$
 \be
 \label{sc13}
  \tex{
 |v(y,\t)|^{p-1}v(y,\t) \to \frac{e_N}{\var(\t)} \,
 \d(y) \quad \mbox{as\,\, $\t \to +\iy$}
 }
  \ee
  in the sense of distributions, where $e_N>0$ is the constant
 \be
 \label{e1}
  \tex{
  e_N= \int\limits_{\ren} W_0^p(\zeta)\, {\mathrm d} \zeta.
   }
   \ee

 Therefore, on this manifold of solutions, the rescaled equation \ef{3.2}
  takes asymptotically the form
  \be
  \label{tt127}
   \tex{
 v_\t= \AAA(v) \equiv -\D^2 v- \frac 14\, y\cdot \n v - \frac
  {N-2}{8}\, v -\frac{e_N}{ \var(\t)} \, \D
 \d(y) +...\forA  \t \gg 1.
 }
  \ee

 Using Lemma \ref{lemSpec2}, we are looking for Type II
 patterns
 of the simplest eigenfunction expansion over generalized Hermite polynomials, i.e., assuming that
  \be
  \label{gg4}
  \tex{
  v_\b(y,\t)= c_\b(\t) \psi_\b^*(y) +... \asA \t \to + \iy.
   }
   \ee
 Actually, this means looking for a kind of stable (or centre)
 subspace behaviour for the equation \ef{3.2}. Then, as usual,
 substituting \ef{gg4} into \ef{tt127} and multiplying by the
 adjoint eigenfunction $\psi_\b$ (recalling the bi-orthonormality
 \ef{Ortog}) yield the following asymptotic equation for
 the expansion coefficient:
  \be
  \label{gg4N}
   \tex{
   \dot c_\b=-\a_\b c_\b +
  \frac { h_\b}{\var(\t)}+...\,,
  }
  \ee
  where  $\a_\b= \frac {2|\b|+N-2}8>0$.
  The crucial coefficient $h_\b$ is calculated as
  follows:
 \be
 \label{h11}
  \tex{
  h_\b=- e_N \langle \D \d, \psi_\b \rangle =
  - e_N \langle  \d, \D \psi_\b \rangle= - e_N (\D \psi_\b)(0)
  \not = 0.
   }
   \ee
 According to \ef{eig1}, the non-vanishing condition \ef{h11}
 imposes essential restrictions on the multi-indexes $\b$ and admissible
 eigenfunctions
 $\psi_\b(y)$ in \ef{eig1},
 for which such blow-up patterns can actually exist. Note that,
 according to \ef{eig1}, the Laplacian in \ef{h11}
  $$
   \tex{
   \D \psi_\b \equiv \frac{(-1)^{|\b|}}{\sqrt{\b !}}\, \big(
    D^2_{y_1 y_1} D^\b F+ D^2_{y_2 y_2} D^\b F +...+ D^2_{y_N y_N} D^\b
    F\big)
    }
    $$
    is indeed a linear combinations of $N$ other eigenfunctions
    of $\BB$. The non-vanishing condition in \ef{h11}
    requires that all the components of the corresponding
    multiindex $\b$ should be even.


 Let us return to the crucial ``dynamical system" \ef{gg4N}, where we show the leading
equation accompanying infinitely many others corresponding to
further stable subspaces. The asymptotic equation contains two
unknowns, the actual coefficient $c_\b(\t)$ and the corresponding
scaling function $\var= \var_\b(\t)$ from \ef{gg4}. Obviously, the
key issue now is to establish an asymptotic relation between them
for $\t \gg 1$, which will allow a proper ``balance" that is
necessary for existence of such a blow-up pattern. On one hand,
this looks like a standard procedure by assuming that the singular
component in \ef{gg4} actually determines the evolution of the
expansion coefficient $c_\b(\t)$. Under this hypothesis, we then
have, in a standard manner, by using \ef{eig1}, for $\t \gg 1$,
 \be
 \label{43}
 \begin{matrix}
  c_\b(\t) \sim \langle \var(\t)W_0(\var^{\frac{p-1}2}(\t) y),
  \psi_\b \rangle =\var(\t) \frac{(-1)^{|\b|}}{\sqrt{ \b !}} \, \langle W_0(\var^{\frac{p-1}2}(\t) y),
  D^\b F \rangle \qquad\qquad
 \ssk\ssk \\
   =
  \var(\t) \frac{1}{\sqrt{ \b !}} \, \langle D^\b_y W_0(\var^{\frac{p-1}2}(\t) y),
   F(y) \rangle. \qquad\qquad
   \end{matrix}
   \ee
 Finally, changing the variable in the last integral by setting $z=\var^{\frac{p-1}2}(\t) y$,
  we arrive at
 the following integral relation:
  \be
  \label{44}
   \tex{
   c_\b(\t) \sim [\var(\t)]^{1+\frac {(p-1)(|\b|-N)}2}
   \int\limits_{\ren} D^\b_z W_0(z) \, F(z
   \var^{-\frac{p-1}2}(\t))\, {\mathrm d}z.
 }
  \ee
 However, resolving uncertainties in \ef{44} as $\t \to +\iy$ is
 not that easy. Moreover, it is not  clear that the projection integral operator onto
 the eigenspace in \ef{44}  gives a correct link between these
 two functions, since, in some cases, extra integrals over subsets in the rescaled
 variables of {\em Outer Regions} (which we do not study here) should be taken into account.
  Recall that even in the second-order case of the
 RD equation \ef{80}, where standard Hermite polynomials and classic spectral theory
 occur, a sufficiently sharp obtaining of all the time factors
 is not always possible, \cite{Fil00}, especially in higher
 dimensions $N \ge 7$, with no results obtained at all.

Therefore, instead of dealing with a singular integral such as in
\ef{44}, we apply another, simpler, but more qualitative and rough
(but sufficient for our goals) method of ``balancing" the
expansions, which, in some cases of blow-up reaction-diffusion
theory, led to rigorous results; cf. various examples in
\cite[Ch.~4-11]{AMGV}. Since, for higher-order parabolic
equations, we do not have any chance of getting more justified
formal expansions, we are allowed to concentrate on a principal
issue of balancing the asymptotic expansion, without trying to
perform further matching of \ef{sc12} with outer regions. This can
be very difficult even for the second-order case \cite{Fil00},
where some parameter ranges require further analysis and even new
ideas.

\ssk

Thus, as usual in blow-up approaches, existence of such a blow-up
pattern requires a certain balance of the two leading terms on the
right-hand side of \ef{gg4N}, i.e., one needs
 \be
 \label{41}
 \tex{
 \a_\b c_\b(\t)  \sim
  \frac { h_\b}{\var(\t)} \LongA  c_\b(\t) \sim \frac{ h_\b
  }
  {\a_\b\var(\t)} \forA \t \gg 1,
  }
  \ee
 where the sign ``$\sim$" assumes omitting other multipliers of
 slower behaviour. Overall, this means that we can use the
 following {\em ansatze} for the expansion coefficient:
  \be
  \label{45}
   \tex{
    c_\b(\t) = \frac {\kappa(\t)}{\var(\t)},
    }
    \ee
where $\kappa(\t)$ is a slow varying function as $\t \to +\iy$ in
comparison with $\var(\t)$. Substituting \ef{45} into
 \ef{gg4N} yields
  \be
  \label{46}
   \tex{
    \frac {\dot{\kappa}}{\var} - \kappa \frac{\dot{\var}}{\var^2}=
    - \a_\b \frac{\kappa}{\var} + \frac{h_\b}{\var}+... \,.
     }
     \ee
This gives the only $\kappa$-independent balance:
 \be
 \label{47}
  \tex{
- \kappa \frac{\dot{\var}}{\var^2} \sim
    - \a_\b \frac{\kappa}{\var} \LongA \frac{\dot{\var}}{\var}
    \sim \a_\b
 \LongA \var=\var_\b(\t) \sim {\mathrm e}^{\a_\b \t}.}
  \ee
 The slow varying function $\kappa(\t)$ cannot be determined from
such a simple matching and requires further difficult asymptotic
analysis of projections like \ef{44} or other approaches.


Thus, up to slower scaling factors, we get a countable family of
such patterns with
 \be
 \label{gg5}
  \var_\b(\t) \sim {\mathrm e}^{\a_\b \t}+... \andA c_\b(\t) \sim {\mathrm
  e}^{-\a_\b \t}+... \forA \t \gg 1, \quad |\b| \ge 0.
  \ee

The expansion \ef{gg4} actually assumes dealing with a 1D
eigenspace, which is the case for $|\b|=0$ only, where $\l_0=0$ is
simple. For any $k=|\b| \ge 1$, a more general, than \ef{gg4}
eigenfunction expansion should be taken into account:
 \be
  \label{gg4NN}
  \tex{
  v_\b(y,\t)= \sum_{|\b|=k} c_\b(\t) \psi_\b^*(y) +...\, ,}
  \ee
 which leads to more difficult dynamical systems for the
 coefficients $\{c_\b(\t)\}_{|\b|=k}$ (to say nothing of the multiple projection
 integrals, which replace
 \ef{44}), but eventually can induce
 more exiting Type II blow-up patterns.

 Overall, bearing in mind the scaling in \ef{sc12}, this yields
 a possibility of constructing
  a countable family of distinct complicated
Type II blow-up structures, where most of them are not radially
symmetric. To reveal the actual space-time and changing sign
structures of such Type II patterns, special matching procedures
apply. In \cite{Fil00}, this analysis has been performed in the
radial geometry for \ef{80}, though
 still no rigorous justification of the existence of
such blow-up scenarios in $\ren$ was achieved. In  \cite{Nai07},
existence of related radial nonnegative blow-up patterns was
encouraged by putting zero Dirichlet data on the boundary of a
shrinking ball. This boundary constraint indeed essentially
simplifies the problem in comparison with those in $\ren$.

Thus,  the first Fourier coefficient in \ef{gg4} or general
expansions on multi-dimensional eigenspaces
 imply a complicated structure of the pattern about the formed
Dirac's  $\d(y)$ according to \ef{sc13}. However, since these
expansions are given by generalized Hermite polynomials
$\{\psi_\b^*\}$, this matching is expected not to impose more
difficulties as those in \cite[\S~4]{Gal5types}.

As a related extension issue beyond blow-up, let us note that the
singular part \ef{sc12}, with the factors \ef{gg5}, creates at
$t=T$ a very weak singularity such that, in the sense of
distributions, its singular part is as follows:
 \be
 \label{99}
  \tex{
  u(x,T) \sim (\d(x))^{\g} \whereA \g= \frac 2{N(p-1)}=\frac{N-2}{2N}< \frac 1{p-1}<1.
  }
  \ee
 Therefore, for such Type II blow-up patterns, by classic parabolic theory,
 for equation \ef{ch6} with such a ``weak singular" data \ef{99} at $t=T$,
  \be
  \label{100}
  \mbox{there exists a unique
 continuation of such blow-up solutions for $t>T$}
   \ee
  locally
 in time ({\em incomplete blow-up}), and such solutions are bounded and classical therein.

Finally, we again comment on the fact that regular {\em stationary
solutions} are key for existence of such Type II blow-up patterns.

 \section{Extra Type II blow-up patterns for the unstable C--H equation:
  linearization about singular steady state and matching}
 \label{S.10}


 Such new Type II blow-up patterns for the semilinear heat equation
 \ef{80} were constructed earlier in \cite{HVsup}; see  \cite{Miz07}
 for extra details.
 We apply this method to the higher-order equation \ef{ch6},
 which will require completely different spectral theory and
 related mathematical tools of matching.

\subsection{Singular stationary solution (SSS)}

Consider the stationary equation in \ef{6.1} in the range
 \be
 \label{5.3}
  \tex{
  p>p_N= \frac N{N-2} \forA N \ge 3.}
   \ee
   Then, as is well known, there exists
  the explicit radial {\em singular steady state} (SSS) of the standard scaling invariant form
   \be
   \label{5.2}
    \tex{
    U(y)=
C_* \, |y|^{-\mu} \whereA \mu= \frac 2{p-1}, \quad
  C_*= D^{\frac 1{p-1}}, \quad
  D= \frac 2{p-1}\big(N-2- \frac 2{p-1}\big)>0.
  }
   \ee

\subsection{Linearization in Inner Region I: discrete
spectrum via  Hardy's inequality}

 We next perform linearization in \ef{3.2} about the SSS by setting:
 \be
 \label{5.4}
  \tex{
v=U+Y \LongA Y_\t= \hat \BB^* Y + \DD(Y), }
 \ee
 where $\DD(Y)$ is a quadratic perturbation as $Y \to
 0$ and
  \be
  \label{5.5}
   \tex{
  \hat \BB^*= \HH^* - \frac 14\, y \cdot \n - \frac 1{2(p-1)} \, I,
  \quad
  \HH^*= -\D^2 - c\D \big({\frac 1{|y|^2}\, I}\big), \,\,\, c={p D}.
   }
   \ee
 Here, $\HH^*$ contains the main singular terms at the origin
 $y=0$.
Similar to Lemma \ref{lemSpec2}, the operator $\hat \BB^*$ at
infinity admits a proper functional setting in the same metric of
$L^2_{\rho^*}$. However, it is also singular at the origin $y=0$,
where its setting depends on the principal part $\HH^*$.

\begin{proposition}
\label{Pr.Har1}
 The symmetric operator $\HH^*$ admits a
Friedrich's self-adjoint extension  with the domain $H^4_0(B_1)$,
discrete spectrum,  and compact resolvent in $L^2(B_1)$, where
$B_1 \subset \ren$ is the unit ball, iff
 \be
 \label{5.6}
  \tex{
  c=p D \le c_{\rm H}= \frac{(N-2)^2}{4}.
  }
  \ee
  \end{proposition}

  \noi{\em Proof.} Indeed, \ef{5.6} is just a corollary of the
  classic
  {\em Hardy} {\em  inequality}
   \be
   \label{5.7}
   \tex{
 \frac {(N-2)^2} {4}  \int\limits_{B_1} \frac {u^2}{|y|^2} \le
 \int\limits_{B_1}
    |\n u|^2 \forA u \in H^1_0(B_1),
    }
    \ee
    where the constant is sharp.
 Therefore, \ef{5.6} implies that the operator $\HH^*$ is
 semi-bounded (in say metric of $H^{-1}(B_1)$), whence the
 necessary properties.
     For compact embedding of the corresponding spaces,
    see Maz'ja \cite[p.~65, etc.]{Maz}. $\qed$

    \ssk

It follows that \ef{5.6} holds in the supercritical
Joseph--Lundgren range
\be
 \label{S2NN}
  \tex{
   p   \ge p_{\rm JL}= 1+ \frac 4{N-4-2\sqrt{N-1}} \forA N \ge
  11,
  }
  \ee
   which, by obvious reasons, coincides with that for \ef{80} in
   \cite{HVsup}.

\subsection{Inner Region I}

Thus, we assume that \ef{5.6} holds and $\s(\hat \BB^*)=\{\hat
\l_k\}$ is discrete, with some eigenfunctions $\{\hat \psi^*_\b,
\, |\b|=k\}$. Furthermore, it is also convenient to assume that
the spectrum is (at least partially) {\em real}. To justify such
an assumption for this non-self-adjoint operator, we rewrite
\ef{5.5} in the form
 \be
 \label{5.11}
  \tex{
 \hat \BB^*= \BB^* - c \D\big({\frac 1{|y|^2} \, I}\big) - \frac 1{2(p-1)}\, I \whereA c=p D
 }
  \ee
 and $\BB^*$ is the standard adjoint operator \ef{4.2} with the real
 spectrum shown in Lemma \ref{lemSpec2}. Actually, this
 means that $\BB^*$ admits a natural self-adjoint representation
 in the space $l^2_{\rho^*}$ of sequences, where it is also sectorial,
 \cite{2mSturm}. Therefore, we claim that the real spectrum of \ef{5.11} can
 be obtained by branching-perturbation theory (see Kato \cite{Kato}) from the spectrum
  $$
 \tex{
  \{\l_\b=- \frac {k}{4}- \frac 1{2(p-1)}, \, k=|\b| \ge 0\}
   }
   $$
    of
 $\BB^*-\frac 1{2(p-1)}\, I$ at $c=0$.
Indeed, this is proved by the following result.

\begin{proposition}
The operators \eqref{5.11}
$$  \tex{
 \hat \BB^*= \BB^* - c \D\big({\frac 1{|y|^2} \, I}\big) - \frac 1{2(p-1)}\, I \whereA c=p D,
 }
  $$
  converge to the operator
   $$
    \tex{
    \BB^* -\frac 1{2(p-1)}\, I,
 }
    $$
   as $c\rightarrow 0$, in the generalized sense of Kato.
\end{proposition}

 \noi{\em Proof.} Indeed, for each $u\in W^{2,2}(B_1)$ we have that
 $$
 \tex{ \left\| \hat \BB^* u- (\BB^*-\frac 1{2(p-1)}\, I)u\right\|_{L^2(B_1)}
 \leq c \left\| \D\big({\frac 1{|y|^2} \, I}\big)u\right\|_{L^2(B_1)}.}
 $$
 Thanks to Hardy's inequality \eqref{5.6}, we arrive at
 $$
 \tex{ \left\| \hat \BB^* u- (\BB^*-\frac 1{2(p-1)}\, I)u\right\|_{L^2(B_1)}
 \leq c K \left\| u\right\|_{W^{2,2}(B_1)},}
 $$
 with $K>0$, a positive constant.
 Therefore, for any $\e >0$, there exists $c_0$ such that
 $$
 \tex{ \left\| \hat \BB^* u- (\BB^*-\frac 1{2(p-1)}\, I)u\right\|_{L^2(B_1)}
 \leq \e \left\| u\right\|_{W^{2,2}(B_1)},}
 $$
 for all $c\in (0,c_0)$ and  $u\in W^{2,2}(B_1)$.
 $\qed$

\vspace{0.2cm}

This shows the convergence of the graphs of the operator $\hat \BB^*$ to the graph of the operator
 $\BB^*$ and, hence, the previous claim is proved.



    \ssk

 Next, the branch must be  extended to $c=p D$, which is
 also a difficult mathematical problem; see \cite[\S~6]{GK1} for some extra details, which are not
  necessary here in such a formal blow-up analysis.

Thus, we fix a certain exponentially decaying pattern in {\em
Inner Region} I:
 \be
 \label{5.12}
 Y(y,\t) = C{\mathrm e}^{\hat \l_\b \t} \hat \psi_\b^*(y)+... \asA
 \t \to +\iy \quad (\hat \l_\b < 0).
  \ee
If there exists $\hat \l=0 \in \s(\hat \BB^*)$, the expansion will
correspond to a  centre subspace  one. Note that \ef{5.12}
includes all the non-radial linearized blow-up patterns.

\subsection{Matching with Inner Region II close to the origin}

In order to match \ef{5.12} with  a smooth bounded flow close to
$y=0$, which we call {\em Inner Region} II, one needs the
behaviour of the eigenfunction $\hat \psi_\b^*(y)$ as $y \to 0$.
To get this, without loss of generality, we assume the radial
geometry. Then,
 the principal operator in the
eigenvalue problem
 \be
 \label{5.131}
  \tex{
  \HH^* \hat \psi^*+... = \l \hat \psi^* \asA y \to 0
  }
  \ee
yields the following characteristic polynomial
 \be
 \label{5.13}
  \hat \psi^*(y) = |y|^\g+... \LongA
 H_c(\g) = (\g-2)(\g-3)[\g^2+(N-2)\g + c]=0.
 \ee
 Obviously, the roots $\g=2$ or 3 are not suitable, so that we
 have
  \be
  \label{91}
   \tex{
\g^2+(N-2)\g + c \LongA \g_\pm=- \frac{N-2}2 \pm \sqrt{
\frac{(N-2)^2}4-c},
 }
 \ee
 which makes sense in the subcritical range $c=pD <
 \frac{(N-2)^2}4$.

 Consider the most interesting critical  case
 \be
 \label{5.14}
  \tex{
c \equiv  p D = c_{\rm H}= \frac{(N-2)^2}{4}.
 }
 \ee
 Then, there exists the double root
 $
 \g_{1,2}= - \frac{N-2}2< 0,
 $
which generates two $L^2$-behaviours:
 \begin{equation}
 \label{3.61}
\hat \psi_1^*(y) = |y|^{-\frac{N-2}2}\ln |y|(1+o(1)) \quad
\mbox{and} \quad \hat \psi_2^*(y) = |y|^{-\frac{N-2}2}
(1+o(1))\quad \mbox{as} \,\,\, y \to 0.
 \end{equation}
Note that $H^1_0$-approximations of $\hat \psi_2^*$ establish that
$c_{\rm H}$ is the best constant in (\ref{5.6}).
 Thus, in $L^2$ in the radial (ODE) setting, the deficiency indexes
 of ${\bf H}^*$ are  $(2,2)$, and
 the straightforward conclusion on
the discreteness of the spectrum of Friedrich's extension of
$\HH^*$ follows, \cite[p.~90]{Nai1}.
Note that
 this leads to the so-called
{\em principal solution} with the minimally possible
 growth at the singular point.

Overall, this gives the following behaviour of the proper
eigenfunctions at the origin:
 \be
 \label{5.15}
 \tex{
 \hat \psi_\b^*(y) =- \nu_\b |y|^{-\frac{N-2}2} +... \asA y \to 0
 \quad (\nu_\b >0 \,\,\, \mbox{are normalization constants}).
  }
  \ee
This allows detection of the rate of blow-up of such patterns by
estimating the maximal value of the expansion near the origin:
 \be
 \label{5.17}
  \tex{
  v_\b(y,\t)= C_*|y|^{-\frac 2{p-1}}- \nu_\b C {\mathrm e}^{\hat
  \l_\b \t} |y|^{-\frac{N-2}2} +... \asA y \to 0 \andA \t \to
  +\iy,
  }
  \ee
  where we observe the natural condition of matching:
   \be
   \label{5.18}
   \nu_\b C>0.
    \ee
 Calculating the absolute maximum in $y$ of the function on the
 right-hand side of \ef{5.18} (this is a standard and justified trick in some R--D
 problems; see e.g., \cite{Dold1})  yields an exponential
 divergence:
 \be
 \label{5.19}
 \tex{
 \|v_\b(\cdot,\t)\|_\iy = d_\b {\mathrm e}^{\rho_\b\t}+...\,
 \whereA \rho_\b= \frac {4|\hat \l_\b|}{(N-2)(p-p_{\rm S})}>0
 \quad \big(p > p_{\rm S}\big),
 }
 \ee
  and $d_\b>0$ are some constants. Depending on the
 spectrum $\{\hat \l_\b<0\}$, \ef{5.19} can determine a countable
 set of various Type II blow-up asymptotics.

 Let us specify more
clearly the necessary matching procedure. In a standard manner,
we return to the original rescaled equation \ef{3.2} and perform
the rescaling in Region II according to \ef{5.19}:
 \be
 \label{5.20}
  \tex{
  v(y,\t) = {\mathrm e}^{\rho_\b \t} w(\xi,s), \quad
  \xi={\mathrm e}^{\mu_\b \t}y, \quad \mu_\b=
  \frac{(p-1)\rho_\b}2, \quad 
  s= \frac 1{(p-1)\rho_\b} \, {\mathrm e}^{(p-1)\rho_\b \t}.
  }
  \ee
 Then $w$ solves the following exponentially perturbed uniformly parabolic
 equation:
  \be
  \label{5.21}
   \tex{
   w_s= - \D^2 w - \D( |w|^{p-1}w) - \frac 1{(p-1)\rho_\b} \, \frac 1
   s\, \big[\big(\frac 14 + \mu_\b\big) \xi \cdot \n w+ \big(
   \frac 1{2(p-1)} + \rho_\b\big) w \big].
    }
    \ee
As above, we arrive at a stabilization problem as $s \to +\iy$  to
a bounded stationary solution, which is widely used in blow-up
applications (see examples in \cite{AMGV}). In general, once the
uniform boundedness of the orbit $\{w(s), \,\, s>0\}$ is
established (an open problem), the passage to the limit in
\ef{5.21} as $s \to +\iy$ is a standard issue of asymptotic
parabolic theory, even for the present higher-order case. Recall
that the limit equation
 $$
w_s= - \D^2 w - \D( |w|^{p-1}w)
 $$
 is a gradient system in $H^{-1}$; cf. \ef{ps1}, \ef{ps2}.

Our blow-up patterns correspond to the stabilization uniformly on
compact subsets:
 \be
 \label{5.22}
 w(\xi,s) \to W(\xi), \,\, s \to +\iy, \,\,\mbox{where} \,\, \D W+ |W|^{p-1}W=0,
  \,\, \xi \in\ren, \,\, W(0)=d_\b,
  \ee
for all admissible $|\b|=0,1,2,...\,$.
As is well-known, for $p \in (1,p_{\rm S})$, the stationary
problem in \ef{5.22} does not admit nontrivial nonnegative
solutions, while for any $p \ge p_{\rm S}$, such solutions always
exist. This is not different from the analysis in \cite{HVsup} for
\ef{80}, so we can omit some details.

\ssk

The supercritical case $p > p_{\rm JL}$ is analyzed similarly,
with some natural changes in asymptotics of eigenfunctions and in
equations such as \ef{5.21}; see \cite[\S~5]{Gal5types}.

Let us comment on an extended semigroup for $t>T$. Since according
to our construction, this Type II blow-up leaves less singular
 final time profile $u(x,T^-)$ than the SSS (see \ef{5.18})),
  the similarity Type I blow-up
 via \ef{72} are expected to be extensible for $t>T$, so that this
 blow-up is expected to be incomplete. However, this does not guarantee uniqueness of such
 an
 extension at all, which is always a hard problem. Moreover, sometimes, for special kinds of
 singularities for nonlinear PDEs the uniqueness problem is not solvable (a so-called
 {\em  principal non-uniqueness}; see an example in
  \cite{GalNDE5}).

\subsection{On related non-radial blow-up patterns}

These can be predicted in a couple of ways. Firstly, one can start
with a non-radial SSS solving the elliptic equation in \ef{6.1},
 if such solutions exist.
    Secondly, under the condition
\ef{5.6}, a non-radial eigenfunction $\psi_\b^*(y)$ (e.g.,
corresponding to an ``angular" logarithmic blow-up swirl obtained
by angular separation of variables, see \cite[\S~3]{Gal5types}) of
$\hat \BB^*$ can be taken into account. Then the matching will assume
using non-radial entire solutions of \ef{5.12}, which  deserves
further  study.


\begin{appendix}
\section*{Appendix a: towards existence of similarity blow-up profiles of (\ref{73})}
 \setcounter{section}{1}
\setcounter{equation}{0}


 \begin{small}

Without loss of generality, we consider the ODE \ef{73} for $N=1$
(similar ideas apply to the radial case for any $N \ge1$ and
$p<p_{\rm S}$):
 \be
 \label{aa1}
   \left\{
    \begin{matrix}
  -f^{(4)}- \frac 14\, y f'- \frac 1{2(p-1)}\, f -
  (|f|^{p-1}f)''=0 \inB \re_+,\ssk\\
    f'(0)=f'''(0)=0, \qquad\qquad\qquad\qquad\qquad\qquad\qquad
 \end{matrix}
 \right.
  \ee
  where we have put two symmetry boundary conditions at the
  origin.
  The strategy of proving existence of a solution of \ef{aa1} is to
  use the 2D asymptotic bundle \ef{f23} to ``shoot" two symmetry
  conditions at $y=0$. This shooting is well defined:

\ssk

  \noi{\bf Claim 1.} {\em For any} $A,C \in \re$, {\em the function} {\em $f=f(y;A,C)$ is well defined for all $y \in
  [0,\iy)$.}

\ssk

   Indeed, this follows from the obvious fact that the
  principal and leading operators,
   $$
   f^{(4)}=- (|f|^{p-1}f)''+... \LongA f''=- |f|^{p-1}f+...
   $$
   do not allow finite-$y$ blow-up.

One can see that, to match two boundary conditions in \ef{aa1},
specific ``oscillatory" properties of solutions $f(y;A,C)$ are
necessary. Then a ``min-max"-like procedure can be applied, when
we first, for a fixed parameter $A$, change $C$ in such a manner
to get existence of
 \be
 \label{aa2}
  \tex{
  C^+(A) = \inf \, \{C \in \re: \quad f'(0;A,C)>0\}.
   }
   \ee
 Existence of such an $C^+(A)$ is guaranteed at least
  for $A \gg 1$ by necessary oscillatory properties of solutions.
We next start to decrease $A$ in such a manner to guarantee that,
for some $A^+$, we obtain the necessary solution:
  \be
  \label{aa3}
  f'''(0;A^+,C^+(A^+))=0 \LongA \exists \,\,\, f(y)=
  f(y;A^+,C^+(A^+)),
   \ee
   but strong oscillatory properties are again required.

\ssk

 \noi{\bf Claim 2.} {\em There exists a subset of solutions  \ef{aa1}, which
 are oscillatory close to the origin.}

\ssk

Note that this is not straightforward: e.g., the bundle \ef{f23}
is not oscillatory at all. Therefore, we have to find oscillatory
structures of solutions, which {\em are not small}. To get a
functional ``topology" of oscillatory solutions, we perform the
similarity scaling (not invariant) in \ef{aa1},
 \be
 \label{aa4}
  \tex{
  f(y)=\e^{- \frac 1{2(p-1)}} g(z), \quad z= \frac y{\e^{1/4}},
  \quad \e>0,
  }
 \ee
 to get the following singularly perturbed ODE for $\e \ll 1$:
  \be
  \label{aa5}
  \tex{
  - \e \, g^{(4)}- \frac 14\, z g'- \frac 1{2(p-1)}\, g -
  (|g|^{p-1}g)''=0.
   }
   \ee
It follows that we can describe a set of solutions which are
oscillatory about the limit profile $g_0(z)$ satisfying
 \be
 \label{aa6}
  \left\{
   \begin{matrix}
- \frac 14\, z g_0'- \frac 1{2(p-1)}\, g_0 -
  (|g_0|^{p-1}g_0)''=0, \quad z>0, \ssk\\
  g_0(0)=1, \quad g'(0)=0. \qquad\qquad\qquad\qquad\qquad
   \end{matrix}
    \right.
    \ee
One can check by maximum principle arguments that such a $g_0(z) >
0$ exists on some interval $z \in [0,z_0)$ and is strictly
monotone there.

We next perform the linearization about $g_0$ by introducing the
new fast variable $Z$:
 \be
 \label{aa7}
  \tex{
  g(z)= g_0(z) + G(Z) \whereA Z= \frac z{\sqrt \e}.
   }
   \ee
 Substituting \ef{aa7} into \ef{aa5} and performing the linearization in the last term, we conclude
  that there exists a subset of solutions satisfying, uniformly on
  compacts in $Z$ the linearized ODE
   \be
   \label{aa8}
    -G^{(4)}_Z- p(g_0^{p}(z)\, G)''_{Z}+...=0,
     \ee
     where we omit further linear and nonlinear terms of the
     order, at least, $O(\e)$.
 Since $g_0(z) = 1+o(1)$ on such compacts, the linearized ODE
admits further simplification:
 \be
 \label{aa9}
-G^{(4)}- p \, G''=0.
 \ee
 It follows that there exists a 2D subset of purely oscillatory
 solutions about $g_0(z)$ with the typical behaviour, for $\e \ll
 1$,
  \be
  \label{aa10}
   \tex{
  g(z) = g_0(z) + B_1 \cos \big( \frac {\sqrt p \,z}{\sqrt {\e}}\big)+
  B_2 \sin \big( \frac {\sqrt p \,z}{\sqrt {\e}}\big)+ ...\, , \quad
  \mbox{where} \quad B_{1,2} \in \re.
}
 \ee
 We stop at this moment our analysis and refer to \cite[\S~4]{EGW}
 for a similar and much more detailed study of such oscillatory
 solutions of the third-order ODE \ef{aa1} for $p=3$.

 Finally, we claim that, using the oscillatory bundle \ef{aa10},
 it is possible to prove existence of a finite limit \ef{aa2} for $|A| \gg 1$.
 Extending this strategy further to get \ef{aa3} is  more
 difficult, but seems doable (in finite time).

 Anyway, the above analysis clearly shows (but does not prove
 completely rigorously) existence of a first blow-up profile
 $f(y)$. What is very difficult and remains entirely open is how
 to catch a possible multiplicity of solutions. It seems that the
 ideas of a $\mu$-bifurcation analysis (see Section \ref{S.6.2} and \cite[\S~4]{BGW1} for
 the RD equation \ef{80}), when \ef{aa2} is replaced
 by the equation
  \be
  \label{aa12}
 \tex{
-f^{(4)}- \mu\, y f'- \frac 1{2(p-1)}\, f -
  (|f|^{p-1}f)''=0 \inB \re_+,
  }
  \ee
changing $\mu$ up to the required $\mu= \frac 14$ and using
the discrete spectrum of the linearized operator (cf. \ef{4.2})
 \be
 \label{BB1}
 \BB^*(\mu)= -\D^2 - \mu y \cdot \n,
  \ee
  are not applicable for the C--H-type nonlinearities in the divergent form.

Finally, let us mention that, for $p \not = p_0=3$, when \ef{aa1}
reduces to the third order, we have not succeeded in obtaining
similarity profiles numerically by solving the ODE. Recall, that
Figure \ref{FF1} was obtained by a PDE numerical modelling.

\end{small}
\end{appendix}

\end{document}